\documentclass[11pt]{amsart}
\usepackage[
text={450pt,575pt},
headheight=9pt,
centering
]{geometry}

\usepackage{latexsym}
\usepackage[all,cmtip]{xy}
\usepackage{times}
\usepackage{braket}
\usepackage[english]{babel}
\usepackage[utf8]{inputenc}
\usepackage{paralist,url,verbatim}
\usepackage{amscd}
\usepackage[T1]{fontenc}
\usepackage{xspace}
\usepackage{comment}
\usepackage{amsmath,amsthm,amssymb,amsfonts}
\usepackage{mathtools}
\usepackage{enumitem}
\usepackage{faktor}
\usepackage{wasysym}
\usepackage[usenames,dvipsnames]{xcolor}
\usepackage[bookmarksopen=true]{hyperref}
\definecolor{darkblue}{RGB}{0,0,160}
\definecolor{darkgreen}{RGB}{0,80,0}
\hypersetup{
colorlinks,
citecolor=black,
filecolor=black,
linkcolor=darkblue,
urlcolor=darkblue,
}
\setlist[enumerate]{labelindent=0pt, labelwidth =0pt, itemindent=*,leftmargin=0pt, font=\upshape}

\newcommand{\cone}{\operatorname{cone}}

\newcommand{\cT}{\mathcal{T}}
\definecolor{cKlaus}{rgb}{0.15,0.40,0.03}
\definecolor{cKLinkGBU}{rgb}{1,0,0}  
\definecolor{cKLink}{rgb}{0.6,0.2,0.3}
\definecolor{cALink}{rgb}{0,0.3,0}
\usepackage[textsize=tiny]{todonotes}

\usepackage{tikz}
\usetikzlibrary{decorations.markings,intersections,positioning,calc}

  \tikzset{mylabel/.style  args={at #1 #2  with #3}{
    postaction={decorate,
    decoration={
      markings,
      mark= at position #1
      with  \node [#2] {#3};
 } } } }

\theoremstyle{plain}
\newtheorem{theorem}{Theorem}[section]

\newtheorem{corollary}[theorem]{Corollary}
\newtheorem{lemma}[theorem]{Lemma}
\newtheorem{proposition}[theorem]{Proposition}

\theoremstyle{definition}
\newtheorem{remark}[theorem]{Remark}
\newtheorem{definition}[theorem]{Definition}

\newtheorem{example}[theorem]{Example}

\newtheorem{convention}[theorem]{Convention}

\newtheorem*{acknowledgement}{Acknowledgement}

\newcommand{\Hom}{\operatorname{Hom}}

\newcommand{\A}{\ensuremath{\mathbb{A}}}
\newcommand{\NN}{\ensuremath{\mathbb{N}}}

\newcommand{\QQ}{\ensuremath{\mathbb{Q}}}
\newcommand{\Q}{\ensuremath{\mathbb{Q}}}

\newcommand{\Rand}[2]{\ensuremath{\partial_{#1}{ #2} }}
\newcommand{\RR}{\ensuremath{\mathbb{R}}}

\newcommand{\Vrtx}{\operatorname{Vert}}
\newcommand{\VrtxZ}{\operatorname{Vert}_{\in\mathbb{Z}}}
\newcommand{\VrtxnZ}{\operatorname{Vert}_{\notin\mathbb{Z}}}
\newcommand{\ZZ}{\ensuremath{\mathbb{Z}}}
\newcommand{\Z}{\ensuremath{\mathbb{Z}}}

\newcommand{\bfa}{\ensuremath{\mathbf{a}}}

\newcommand{\bfc}{\ensuremath{\mathbf{c}}}
\newcommand{\bfk}{\ensuremath{\mathbf{k}}}

\newcommand{\bfs}{\ensuremath{\mathbf{s}}}
\newcommand{\bft}{\ensuremath{\mathbf{t}}}
\newcommand{\bfx}{\ensuremath{\mathbf{x}}}

\newcommand{\ceil}[1]{\lceil #1 \rceil}

\newcommand{\conv}{\operatorname{conv}}

\newcommand{\im}{\mbox{\rm Image}}

\newcommand{\kenditem}{\vspace{-1ex}\end{itemize}}

\newcommand{\kitem}{\begin{itemize}\vspace{-2ex}}

\newcommand{\ko}{\overline}

\newcommand{\ks}{\scriptstyle}

\newcommand{\ku}{\underline}
\newcommand{\la}{\ensuremath{\lambda}}

\newcommand{\surj}{\longrightarrow\hspace{-1.5em}\longrightarrow}
\newcommand{\then}{\Rightarrow}

\newcommand{\wt}[1]{\widetilde{#1}}

\newcommand{\wtR}{\ensuremath{\widetilde{R}}}

\newcommand{\dual}{^{\vee}}

\newcommand{\kT}{{T}} 
\newcommand{\kS}{{S}} 
 
\newcommand{\ksum}{{a}} 
 
\newcommand{\bound}{{\partial}} 
\newcommand{\height}{{\lambda}} 
\newcommand{\ktT}{{\smash{\wt{T}}}} 
\newcommand{\ktS}{{\smash{\wt{S}}}} 
\newcommand{\ktt}{{\wt{t}}} 
\newcommand{\kts}{{\wt{s}}} 
\newcommand{\kpi}{{\pi}}

\newcommand{\kQuot}{{M}} 
\newcommand{\kQuotR}{{M_{\RR}}} 
\newcommand{\kQuotD}{{N}} 
\newcommand{\kQuotRD}{{N_{\RR}}}

\newcommand{\kP}{P}

\newcommand{\kiota}{{\iota}} 
 
\newcommand{\kpr}{{p}}

\newcommand{\kR}{{R}} 
\newcommand{\kMC}{{C}} 
 
\newcommand{\kMV}{{C^{\mbox{\tiny lin}}}}

\newcommand{\kMT}{{\mathcal T}} 
\newcommand{\kMTD}{{\mathcal T^*}} 
\newcommand{\kMTZ}{{\mathcal T_{\ZZ}}} 
\newcommand{\kMTZD}{{\mathcal T_{\ZZ}^*}}

\newcommand{\kPm}{{m}} 
\newcommand{\cedges}{{\operatorname{edge}}}
\newcommand{\kPr}{{r}} 
\newcommand{\se}{{short}} 
 
\newcommand{\vast}{{\ensuremath{{v_\ast}}}}

\newcommand{\kL}{{L}} 
 
\newcommand{\gens}{{g}}

\newcommand{\oneone}{{(\ku{1};\ku{1},\ku{0})}} 
\newcommand{\tX}{{{\widetilde{X}}}}
 
\newcommand{\baseMfullPre}{{{\mathcal M}}} 
\newcommand{\baseMkick}{{\ko{{\mathcal M}}}}

\newcommand{\spec}{\operatorname{Spec}}

\newcommand{\un}{\underline}
\renewcommand{\span}{\operatorname{Span}}

\newcommand{\lan}{\langle}
\newcommand{\ran}{\rangle}
\newcommand{\id}{\operatorname{id}}

\newcommand{\too}{\ensuremath{\longrightarrow}}

\newcommand{\tail}{\operatorname{tail}}
\newcommand{\kst}{\,:\;}
\newcommand{\kc}{c}
\newcommand{\keta}{\eta}
\newcommand{\ketaZ}{\eta_\ZZ}
\newcommand{\kbb}{\bullet}

\usepackage{tikz-cd}

\newcommand{\inter}{\operatorname{int}}

\newcommand{\cI}{\mathcal{I}}
\newcommand{\cJ}{\mathcal{J}}

\newcommand{\cM}{\mathcal{M}}

\newcommand{\cO}{\mathcal{O}}
\newcommand{\cR}{\mathcal{R}}
\newcommand{\bfT}{\bold{T}}

\newcommand{\bo}{\partial}

    \renewcommand{\la}{\height}
     \newcommand{\tla}{\widetilde{\lambda}}
        
        \renewcommand{\bo}{\bound}
         \newcommand{\tbo}{\widetilde{\partial}}

     \newcommand{\mM}{\baseMkick}

\newcommand{\eps}{\epsilon}
\newcommand{\lam}{\lambda}
  
   \newcommand{\h}{\operatorname{Hom}}
   \renewcommand{\AA}{\mathbb{A}}

\newcommand{\smallestint}[1]{\lceil #1 \rceil}

\newcommand{\roundup}[1]{\smallestint{#1}}

\newcommand{\kteta}{\widetilde{\eta}}
\newcommand{\ktetaZ}{\widetilde{\eta}_\ZZ}
\newcommand{\wteta}{\widetilde{\eta}}
\newcommand{\spann}{\operatorname{Span}}
\newcommand{\Spec}{\operatorname{Spec}}

\newcommand{\te}{\widetilde{\eta}_\ZZ}

\newcommand{\toric}{{\mathbb T \!\mathbb V}}
\newcommand{\kLp}{\widetilde{t}}

\begin{document}
\title{Versality in toric geometry} 

\author[K.~Altmann]{Klaus Altmann
}
\address{Institut f\"ur Mathematik,
FU Berlin,
Arnimalle 3,
D-14195 Berlin,
Germany}
\email{altmann@math.fu-berlin.de}
\author[A.~Constantinescu]{Alexandru Constantinescu}
\address{Institut f\"ur Mathematik,
FU Berlin,
Arnimalle 3,
D-14195 Berlin,
Germany}
\email{aconstant@zedat.fu-berlin.de}
\author[M.~Filip]{Matej Filip}
\address{Laboratory of applied mathematics and statistics,
University of Ljubljana,
Tr\v{z}ai\v{s}ka cesta 25,
SI-1000 Ljubljana,
Slovenia}
\email{matej.filip@fe.uni-lj.si}

\maketitle

\begin{abstract}
We study deformations of affine toric varieties.
The entire deformation theory of these singularities
is encoded by the so-called versal deformation.
The main goal of our paper is to construct the homogeneous part of some
degree $-R$ of this,
i.e.\ a maximal deformation with prescribed tangent space $T^1(-R)$
for a given character $R$.
To this aim we use the polyhedron obtained by cutting the rational cone 
defining the affine singularity with the hyperplane defined by $[R=1]$.
Under some length assumptions on the edges of this polyhedron,
we provide the versal deformation for primitive degrees $R$.
\end{abstract}

\section{Introduction}
Understanding the deformation theory of a toric variety $X$ and its boundary $\partial X$ is useful for several classification projects.
For example, smoothings of such singularities are used to  compactify moduli spaces of surfaces of general type.
In line with the recent interest in classifying Fano manifolds using Mirror Symmetry \cite{coa,pri,mojcor},
it is conjectured that all low dimensional smooth Fano varieties  can be degenerated to a singular Fano toric variety  \cite{cor}.
By the comparison theorem of Kleppe from \cite{Kle79}, see also \cite[Section 2.1]{chr-ilt} for an overview, understanding deformations of \emph{affine} toric varieties implies understanding deformations of projective toric varieties as well.

The versal base space of an affine toric singularity inherits a torus action, and thus a lattice grading.
Our aim is to construct a maximal deformation in a given primitive degree $-R$.
The rational cone defining the toric singularity, together with the degree $-R$,
can be entirely reconstructed from a rational polyhedron.
For isolated Gorenstein toric singularities the whole versal deformation is concentrated in a single degree (the ``Gorenstein degree'').
Assuming also smoothness in codimension two, the corresponding polyhedron is a lattice polytope with primitive edges. 
The versal deformation for such  toric singularities was obtained in \cite{alt}.
In this paper we drop both the Gorenstein and the smoothness in codimension two assumptions, so the versal base space may have several non-trivial graded components.
As a special case, we obtain yet another point of view for the  deformations of 2-dimensional cyclic quotient singularities (\cite{KSB88,Chr91,Ste91}).\\

We work over an algebraically closed field $k$ of characteristic $0$. Let $\kQuotD$ and $\kQuot$ be dual lattices
and let $\sigma \subseteq (\kQuotD\oplus \ZZ)\otimes_\ZZ\RR$ be a polyhedral cone.
The associated affine toric variety is
\[
   X=\toric(\sigma)=\spec k[S],
\]
where $\kS=\sigma^\vee\cap (\kQuot\oplus\ZZ)$.
The tangent space of the deformation functor of $X$ is
\(
  T^1_X=\operatorname{Ext}^1_{\cO_X}(\Omega^1_X,\cO_X).
\)
It is a $k$-vector space with an $M$-grading induced by the torus action.
For every $R\in M\oplus \ZZ$ we denote by $T^1_X(-R)$ the graded component of $T^1_X$ of degree $-R$.
Taking a cross-cut of $\sigma$ with the affine hyperplane $[R=1]=\{a\in \sigma \kst \lan a,R \ran=1\}$ we obtain a rational polyhedron $\kP$.
Our goal is to start from a polyhedron $P$ as above and construct a maximal deformation for $X$ in degree $-R$.

Our approach is to use the polyhedron $P$ to construct a pair of  monoids $\ktT\subseteq \ktS$ which fit into the following Cartesian diagram\\
\vspace{-1.1em}
\begin{equation}\label{eq tikz b}
\begin{tikzcd}
X=\spec k[S]  \arrow[d,"R"]\arrow[hookrightarrow]{r}
  & \spec k[\ktS] \arrow[d,"\wt{R}"]
\\
\AA^1_k=\spec k[\NN]  \arrow[hookrightarrow]{r}{i} &
\spec k[\ktT]
\end{tikzcd}
\end{equation}
with $\dim \spec k[\ktT]=\dim_kT^1_X(-R)+1$ and  the property that $k[\ktS]$ is a flat $k[\ktT]$-module. 
The monoid $\ktT$ is a generalization of the Minkowski scheme of a lattice polytope and
the monoid $\ktS$ is a generalization of the monoid corresponding to a tautological cone, see \cite{alt}.
The main result of \cite{a} is that the pair $(\ktT,\ktS)$ is a universal extension of the pair $(\NN\!\cdot\!R,~S)$.
Our hope is that this universal extension allows one to construct the versal deformation of $X$.\\

 Here is the idea how to produce a deformation diagram for $X$ from the diagram~(\ref{eq tikz b}).
Assume that $\ktT$ is generated in degree 1, i.e.\ by elements 
$t_0,\ldots,t_\gens\in\ktT$ mapping to $1$
via the map $\ktT\surj\NN$ from \eqref{eq tikz b}. 
From the above assumption we obtain an embedding
$\Spec k[\ktT]\hookrightarrow \Spec k[u_0,\ldots,u_\gens]=\A^{\gens+1}_k$ 
such that the composition 
\(
\Delta:\A^1_\gens=\Spec k[t]\hookrightarrow\Spec k[\ktT] \hookrightarrow \A^{\gens+1}_k
\)
equals the diagonal morphism $1\mapsto\ku{1}$. This means that the corresponding map of $k$-algebras $k[u_1,...,u_g]\to k[t]$ is given by $u_i\mapsto t$.
In particular, this embedding is linear, and we may consider 
the quotient 
\[
\A^{\gens+1}_k \stackrel{\ell}{\surj}
\A^{\gens+1}_k/\Delta:=\A^{\gens+1}_k/(k\cdot\ku{1})=
\Spec k[u_i-u_j\kst 0\leq i,j\leq \gens].
\]
For any given closed subscheme $\baseMkick\subseteq \A^{\gens+1}_k/\Delta$ with 
\(
\baseMfullPre:=\ell^{-1}(\baseMkick)\subseteq\Spec k[\ktT]\subseteq \A^{\gens+1}_k
\)
we obtain the following commutative diagram:
\begin{equation}
  \label{eq imp imp dia}
  \begin{tikzcd}
    X \arrow[hook,r] \ar[d]{}{R}
    & 
    \tX \arrow[hook,r] \arrow[d]{}{\wt{R}}
    &
    \Spec k[\ktS]\arrow[d]{}{\wt{R}}
    \\
    \A^1_k \arrow[hook,r] \arrow[d]
    & 
    \baseMfullPre \arrow[hook,r] \arrow[d]{}{\ell}
    & 
    \Spec k[\ktT] \arrow[hook,r] \arrow[Rightarrow,dl,"\mbox{\tiny maximal}"]
    & 
    \A^{\gens+1}_k\ar[d,"\ell"]
    \\
    0 \arrow[hook,r]
    & 
    \baseMkick \arrow[hook,rr]
    &&
    \A^{\gens+1}_k/\Delta.
  \end{tikzcd}
\end{equation}
The double arrow
indicates that there is a maximal closed 
subscheme $\baseMkick\subseteq\A^\gens_k$ meeting the requirement
$\ell^{-1}(\baseMkick)\subseteq \spec k[\ktT]$. 
The point of the whole construction is the direct interplay between the 
schemes $\spec k[\ktT]$ and $\baseMkick$ -- both refer to different moduli 
problems.
While we will see in Subsection \ref{defHyperplaneSections} that
$\spec k[\ktT]$ is a base space for a deformation of $R^{-1}(0)$, we obtain with
$\baseMkick$ a base space for a deformation of $X$. Indeed, this is a
consequence of the following two facts:
first, the map $\wt{R}:\tX\to\baseMfullPre$ inherits flatness from $\wt{R}:\spec k[\ktS]\to k[\ktT]$. 
Second, since $\baseMfullPre=\ell^{-1}(\baseMkick)$ is a full preimage,
the lower left square in  diagram~(\ref{eq imp imp dia}) is Cartesian with a flat
projection $\ell:\baseMfullPre\surj\baseMkick$. 
We define  $\tilde{X}:=\widetilde{R}^{-1}(\cM)$.

The full details for this are given in Section \ref{sec vers}. 
The main result of this paper is the following theorem.
\begin{theorem}\label{main con intro}
For all compact edges $d$ of $P$ assume that the sub-monoid $\ktT_d\subset \ktT$ is generated by degree $1$ elements. Then the maximal $\baseMkick\subseteq\A_k^\gens$ with 
$\ell^{-1}(\baseMkick)\subseteq \spec k[\ktT]$ 
yields the deformation diagram, which is maximal with prescribed tangent 
space $T^1(-R)$. That is, the family $\tX\to \mM$ cannot be extended to a 
larger deformation of $X$ without enlarging the ambient 
linear space of the base.
\end{theorem}

Note that Theorem \ref{main con intro}  has been shown in \cite{alt} and \cite{budapest} for the special case of $X$ lacking singularities in codimension
two, which is a very special case of $\ktT_d$ being generated by degree 1 elements for all compact edges $d$ of $P$, see Section \ref{sec vers}. Our main result holds more general and is obtained with different techniques than in \cite{alt} and \cite{budapest}.\\

The first part of the paper (Sections~\ref{sec the main monoids}~to~\ref{sec flat})  focuses on the monoids $\ktT$ and $\ktS$.
The main results there are the explicit descriptions of the generators of $\ktT$  and $\ktS$ (Proposition~\ref{prop generators t} and Corollary~\ref{cor gen ts}, respectively), and of the relations among them (Section~\ref{r:projectionOfSum} and Prop~\ref{prop tildelam}).
An important feature is that the generators of $\ktT$ can be computed knowing only the compact edges of $P$.
In  Section~\ref{sec vers} we return to algebraic geometry, and introduce our main result: Theorem~\ref{th main}.
 Sections~\ref{sec loop}~to~\ref{sec-clusterCoord} are dedicated to the proof of the main result, which is obtained by proving that the obstruction map is injective.\\

\begin{acknowledgement}
\small 
We are greatful to Alessio Corti for his interest in this work and many useful conversations. 
\end{acknowledgement}

\section{Preliminaries}\label{sec the main monoids}
In this section we recall the construction of  $\ktT\subseteq \ktS$ from \cite{a}. 
For our main result we will then exploit new properties of the monoids $\ktT$ and $\ktS$:
their generators (cf. Section~\ref{sec sem ktt}), their relation to flatness (cf. Section~\ref{sec:Diagrams})
and the syzygies of the corresponding semigroup rings (cf. Section~\ref{sec flat}). 

\subsection{The setup}\label{sec setup}
Throughout the paper  $\kQuotD$ is a lattice of finite rank, that is $\kQuotD\simeq\ZZ^n$ for some $n\in\NN$, and $M=\Hom_\ZZ(N,\ZZ)$  the dual lattice of $\kQuotD$. Let  $\kP\subseteq \kQuotRD :=\kQuotD\otimes\RR$ be a rational, convex polyhedron.
We  embed $\kP$ in an affine hyperplane of height one of $\kQuotRD\oplus\RR$ and take the cone over it:
\[
  \sigma:=\cone(\kP,1)\subseteq \kQuotRD\oplus\RR.
\]
We define the monoid $\kS:=\sigma^\vee\cap (\kQuot\oplus\ZZ)$.
This monoid contains the lattice-primitive element $\kR=(\ku{0},1)$ spanning the discrete ray $\kT:=\NN\cdot\kR$.
Thus, starting from $\kP$ we construct the pair of monoids $\kT\subseteq\kS$.
Our objective is to study the deformations of the affine toric variety 
\[
  X:=\toric(\sigma):=\spec k[S].
\]
To this aim, we will use co-Cartesian extensions, cf.~\cite[Definition 3.1]{a}.
The main result of~\cite{a} was the construction of a universal co-Cartesian extension 

\[
\begin{tikzcd}
   \ktT \arrow[r,hook]\arrow[d,swap]{}{\kpi_{\kT}}
   &
   \ktS \arrow[d]{}{\kpi_{\kS}}
   \\
   \kT \arrow[r,hook]
   &
   \kS.
 \end{tikzcd}
\]

The polyhedron $\kP$ may not be bounded, meaning that its tail cone
\[
\tail(\kP):=\{a\in\kP-\kP\kst a+\kP\subseteq\kP\}
\]
may be not trivial.
Every element $\kc\in\kQuotR$ is a linear form on $\kQuotRD$, and
$\kc$ is bounded below on $\kP$ if $\kc\in\tail(\kP)\dual\subseteq\kQuotR$. 
It is easy to see that the minimum is achieved at some vertex.
For every $\kc\in\tail(\kP)\dual$  we choose and fix one such vertex $v(\kc)$.
While this choice is not unique, the value of the upcoming numbers
$\keta(\kc)\in\QQ$ and $\ketaZ(\kc)\in\ZZ$ will not depend on it.
\begin{definition}\label{d:eta(c)}
For every linear form $\kc\in \tail(\kP)^\vee$ define
\[
  \keta(\kc) :=-\min_{v\in\kP}\braket{v,\kc}=
-\langle v(\kc),\kc\rangle \in\RR. 
\]
This is not always an integer, and we denote the round up to the next integer by $\ketaZ(\kc):= \ceil{\keta(\kc)}\geq\keta(\kc)$.
\end{definition}
Note that $\sigma^\vee=\{[\kc,\keta(\kc)]\kst \kc\in \tail(\kP)^\vee\} + \RR_{\geqslant 0}\cdot[\un{0},1].$
The Hilbert basis of $\kS$ has the form
\begin{equation}\label{eg hilbbas}
\big\{s_1=[c_1,\ketaZ(c_1)],\dots,s_r=[c_r,\ketaZ(c_r)],R:=[\un{0},1]\big\},
\end{equation}
with uniquely determined elements $\kc_i\in \tail(\kP)^\vee\cap\kQuot$.

\subsection{Short edges}
\label{newAmbient}
We denote the set of vertices and the set of compact edges of the polyhedron $\kP$ by
\[
\Vrtx(\kP)=\{v^1\dots,v^\kPm\}
\hspace{1em}\mbox{and}\hspace{1em}
\cedges(\kP):=\{d^1,\ldots,d^\kpr\},
\]
respectively.
If an edge $d^\nu$ connects the vertices $v^i,v^j$, then we will
also denote it by $d^\nu=d^{ij}=[v^i,v^j]$. 
Alternatively, we might equip it with an orientation
by either understanding it as a vector $d^{ij}=v^j-v^i\in\kQuotRD$
or as a half open segment $d^{ij}=[v^i,v^j)$ which is, of course, no longer
compact.

\begin{definition}
\label{def-shortEdge}
To each bounded half open edge $d=[v,w)$ of $\kP$
we associate the positive integer
\[
g_d := \min\{g\in\ZZ_{\geqslant 1}\kst 
\text{the {\em affine line} through }gv \text{ and } gw 
\text{ contains lattice points}\} \geq 1.
\]
We call $d=[v,w)$ a \emph{short half open} edge if
\[
\;\#\{g_d\cdot [v,w)\cap\kQuotD\}\leq g_d-1.
\]
Moreover, we call $d=[v,w]$ a {\em\se} edge if both $[v,w)$ and $[w,v)$ are {\em\se} half open edges.
\end{definition}

In particular, the vertex $v$ of a \se\ half open edge $[v,w)$
 never belongs to the lattice $\kQuotD$. Moreover,
if at least one of the half open edges $[v,w)$ or $[w,v)$ is \se,
then $\ell(w-v)<1$ where $\ell$ denotes the lattice length
-- this is defined as the homogeneous function on $\kQuotRD$ such
that any primitive element of $\kQuotD$ has lattice length one.

\begin{example}
  \label{ex-CQS}
  \begin{enumerate}[label={\arabic*)}, ref ={\arabic*.}]

\item In the one-dimensional case, that is when $\kP=d=[v,w]\subset \RR$ we always have $g_d=1$. The edge $d$ is \se\ if and only if $[v,w]\cap \ZZ=\emptyset$. In particular, the edge $d=[-\frac{1}{m},\frac{1}{n}]\subset\RR$ with $n,m\in \NN$ is never short. 
  
\item \label{item:CQS2} Take $d=\big[(-\frac{1}{6},\frac{1}{2}),\;  (\frac{2}{3},\frac{1}{2})\big]\subset \RR^2$. 
We need to multiply $d$ with $2$ to produce lattice points on the affine line and thus $g_d=2$.
Since $\#\{g_d\kP\cap\kQuotD\}=2$ we see that
both half open edges are not \se.
\item  Take $d=\big[(\frac{1}{2},1),(\frac{3}{4},\frac{5}{4})\big]\subset \RR^2$. Also in this case we have $g_d=2$, so $g_d\cdot d = \big[(1,2),(\frac{3}{2},\frac{5}{2})\big]$, which contains exactly one lattice point. So both half-open edges are short.
\end{enumerate}
\end{example}

It is well-known that the set of Minkowski summands of scalar multiples
of $\kP$ carries the structure of a convex, polyhedral cone
$\kMC(\kP)$,
i.e.\ each $\xi\in\kMC(\kP)$ represents a Minkowski summand
$\kP_\xi$, see \cite[Section 2.2]{alt}. Note that
$\kMV(\kP):=\kMC(\kP)-\kMC(\kP)\subseteq\RR^\kpr$ is a linear subspace 
with coordinates $t_{\nu}$ encoding the dilation of the compact edges.
It is defined by the equations
\begin{equation}\label{2 face equation}
\sum_{d^\nu\in \eps}\delta_\eps(d^\nu)\cdot t_{\nu}\cdot d^{\nu}=0
\end{equation}
where $\eps\leq \kP$ runs through all compact 2-dimensional faces
and $\delta_\eps(d^\nu)\in\{0,1,-1\}$ is chosen such that the edges 
$\delta_\eps(d^\nu)\cdot d^{\nu}\in\kQuotRD$ form a cycle 
along the boundary of $\eps$.

\begin{definition}\label{def-TP}
If $t_{ij}$ denotes the dilation factor for the compact edge $[v^i,v^j]\leq\kP$
and $s_i$ is the coordinate on $\RR^\kPm$ referring to the vertex $v^i$,
then we define
\[
\textstyle
\kMT(\kP):=\big\{ (t,s) \in \kMV(\kP)\oplus \RR^\kPm \kst 
\begin{array}[t]{ll}
s_i = 0 &\mbox{if} \hspace{0.5em} v^i\in\kQuotD,
\\
s_i = s_j & \mbox{if}\hspace{0.5em} [v^i,v^j]\leq\kP \text{ with }
[v^i,v^j]\cap\kQuotD=\emptyset, \text{ and }
\\
s_i=t_{ij} & \mbox{if}\hspace{0.5em} [v^i,v^j) 
\text{ is a half open \se\ edge} \big\}.
\end{array}
\]
\end{definition}

Note that the vector space $\kMT(\kP)$ contains a distinguished element
$\oneone=[\kP]$ which is defined by $s_i:=0$ for $v^i\in\kQuotD$ and
$s_j:=1$ and $t_{ij}:=1$ for all remaining coordinates.
In the upcoming sections we will often deal with the dual vector space
$\kMTD(\kP)$, where elements $s_i,t_{ij}\in\kMTD(\kP)$ form 
a generating set. We could easily omit the elements 
$s_i=0$ for $v^i\in\kQuotD$. However, while they are just zero, 
there existence will simplify some formulae. Let
\begin{equation}\label{def map p}
\pi:\cT^*(P)\to \RR
\end{equation}
be the map that sends the generators 
$t_{ij}\in\kMTD(\kP)$ to $1$ 
and $s_i\in\kMTD(\kP)$ to $1$ or $0$ depending on  
$v^i\notin N$ or $v^i\in N$, respectively. Note that this map is well-defined. 

\begin{proposition}
\label{prop-compTOne}
For any rational polyhedron $\kP$ and for  $R=[\un{0},1]\in\kQuot$ we have
\[
  T_{X_\sigma}^1(-\kR)=\faktor{\big(\kMT(\kP)\otimes_{\RR}k\big)}{k\cdot\oneone}.
\]

\end{proposition}

\begin{proof}
Essentially, this corresponds to \cite[Theorem 2.5]{ka-flip}.
One has just to check that the equations called $G_{jk}$ in
\cite[(2.6)]{ka-flip} coincide with those 
in the definition of the $\RR$-vector space $\kMT(\kP)$.
\end{proof}

\subsection{The lattice structure in $\kMT(\kP)$}
\label{ssec:theLatticeStructureInT(P)}
\begin{definition}\label{d:latticeInT} 
We define the subgroup $\kMTZ(\kP)\subset\kMT(\kP)$ by
\[
\begin{array}{rcl}
  (\bft,\bfs)\in \kMTZ(\kP) &:\iff&
  \begin{cases}
    s_i\in \ZZ &\forall~ v^i \in \Vrtx(\kP),\\
    (t_{ij}-s_i)v^i - (t_{ij}-s_j)v^j 
     \in \kQuotD &\forall~ [v^i,v^j]\in \cedges(\kP).
  \end{cases}
\end{array}
\]
\end{definition}
\noindent Clearly $\kMTZ(\kP)$ is a subgroup of $\kMT(\kP)$, thus it is torsion-free and Abelian. Moreover, it is easy to see that it is a free Abelian group satisfying
$\kMTZ(\kP)\otimes_{\ZZ}\RR=\kMT(\kP)$, see also \cite[Lemma 5.14]{a}.
Using the dual lattice $\kMTZD(\kP)$, the two conditions of Definition \ref{d:latticeInT} can be rephrased as:
\begin{eqnarray*}
  s_i&\in&\kMTZD(\kP),\mbox{ and}\\
\kL_{ij}:=(t_{ij}-s_i)\otimes v^i -
(t_{ij}-s_j)\otimes v^j
&\in& \kMTZD(\kP)\otimes_{\ZZ}\kQuotD.
\end{eqnarray*}

\subsection{The  main monoids}\label{subSectMainMon}
For the upcoming constructions we need to choose and fix a reference vertex $\vast\in\kP$.
We establish the following convention, which may require  shifting
$\kP$ by a lattice vector. 
\begin{convention}\label{conv:v*=0}
Whenever $\vast\in\kP$ belongs to the lattice $N$, we assume that $\vast=0$.
\end{convention}
For every  $\kc\in \tail(\kP)^\vee$ we choose a path 
$
\vast = v_0, v_1, \ldots ,v_k = v(\kc)
$
along the compact edges of $\kP$, and split $-\keta(\kc)$
from Definition~\ref{d:eta(c)} as a sum in the following way:
\[
  \textstyle
  -\keta(\kc)=\braket{ v(\kc),\,\kc}=\braket{\vast,\,\kc} + \sum_{j=1}^k \braket{ (v_{j}-v_{j-1}),\,\kc}.
\]
This leads us to the next definition.
\begin{definition}\label{d:etatilde(c)}
For every $\kc\in\tail(\kP)^\vee$, we define $\kteta(\kc)\in \cT^*(\kP)$ as
\[
\textstyle
\wteta(\kc):= -\braket{\vast,\,\kc}\cdot s_\vast - 
\sum_{j=1}^k  \braket{(v_{j}-v_{j-1}),\,\kc} \cdot t_{j-1,\,j}.
\]
Note that the first summand, i.e.\ 
the $s_\vast$ part, vanishes if $\vast\in \kQuotD$. 
\end{definition}
It is easy to see that
the definition of $\wteta(\kc)\in \kMTD(\kP)$ does neither depend on the
choice of the vertices $v_{*}$ and $v(\kc)$, nor on the choice of the path connecting
$\vast$ and $v(\kc)$.
Note also that, due to Convention~\ref{conv:v*=0}, $\wteta(\kc)$ is always a lifting of
$\keta(\kc)$ via the map $\kpi$ from the equation \eqref{def map p}.

\begin{definition}
\label{def-etsTildeZ}
For every $\kc\in \tail(\kP)^\vee\cap\kQuot$ we define 
\[
  \textstyle
\ktetaZ(\kc):=
\wteta(\kc) + \big(\ketaZ(\kc)-\keta(\kc)\big)\cdot s_{v(\kc)}
=
\ketaZ(c)\cdot s_{v(c)}+\sum_{j=1}^kL_{j-1,j}(c)
\;\in\cT^*_{\ZZ}(\kP),
\]
where $L_{j-1,j}(c):=\lan L_{j-1,j},c\ran$.
Moreover, for
$\kc_1,\kc_2\in \tail(\kP)^\vee\cap\kQuot$ we measure convexity via
\[
\ktetaZ(\kc_1,\kc_2) := 
\ktetaZ(\kc_1)+\ktetaZ(\kc_2) - \ktetaZ(\kc_1+\kc_2)
\;\in\cT^*_{\ZZ}(\kP).
\]
\end{definition}
\noindent The main monoids which provide a universal extension in \cite[Theorem 8.2]{a},
and which we will analyse in order to produce a maximal deformation, are the following.
\begin{definition}\label{def:upperSemigroups}
  For every rational polyhedron $\kP$, in the above notation, define:
\begin{eqnarray*}
\ktT &:=& 
\spann_{\NN}\{[0,\ktetaZ(\kc_1,\kc_2)]\kst \kc_1,\kc_2\in \tail(\kP)^\vee\cap \kQuot\}\\
  \ktS &:=&
           \ktT + \spann_{\NN}\{[\kc,\ktetaZ(\kc)]\kst \kc\in \tail(\kP)^\vee\cap \kQuot\},
\end{eqnarray*}
with $\ktT\hookrightarrow \ktS\subset M\oplus \cT^*_{\ZZ}(\kP)$.
These two objects fit in the following diagram
\[
  \xymatrix{
    \ktT~ \ar@{^{(}->}[r] \ar@{->}[d]_{\kpi_{\kT}}
    &
    \ktS \ar@{->}[d]^{\kpi_{\kS}}\\
    \NN \ar@{^{(}->}[r] &  S
  }
\]
with vertical maps induced by $t_{\kbb,\kbb},s_v\mapsto 1$ for $v\not\in N$ and $s_v\mapsto 0$ for $v\in N$.
We call  $\kpi_T:\ktT\to T=\NN$ the \emph{degree map}. 
\end{definition}

In particular, for $\kc_1,\kc_2,\kc\in \tail(\kP)^\vee\cap\kQuot$ and $\kLp\in \ktT$, we have
\begin{eqnarray*}
  \ktetaZ(\kc_1,\kc_2)&\stackrel{\kpi_\kT}{\longmapsto}& \ketaZ(\kc_1)+\ketaZ(\kc_2)-\ketaZ(\kc_1+\kc_2)\in \NN,\text{~~and}\\
  \kLp+[\kc,\ktetaZ(\kc)]&\stackrel{\kpi_\kS}{\longmapsto}&[\un{0}, \kpi_\kT(\kLp)]+[\kc,\ketaZ(\kc)]\in S.
\end{eqnarray*}

\section{Explicit generators of $\ktT$}\label{sec sem ktt}

We start this section  by analysing the monoid $\ktT$.
We denote the set of non-lattice vertices of $\kP$, that is the set of vertices that are not contained in $N$, with $\VrtxnZ(\kP)$.
We will write $\VrtxZ(\kP)$ for the set of lattice vertices, so $\VrtxnZ(\kP)=\Vrtx(\kP)\setminus\VrtxZ(\kP)$.
Moreover, for real numbers $z\in \RR$ we will quite often use the following notation:
\begin{equation}\label{eq fun f}
\{z\}:=\roundup{z}-z.
\end{equation}
In particular, the following lemma trivially holds.
\begin{lemma}\label{lem trivial}
For each $z_1,z_2\in \RR$ there is either
$\{z_1+z_2\}=\{z_1\}+\{z_2\}$	or $\{z_1+z_2\}+1=\{z_1\}+\{z_2\}$.
\end{lemma}

Let $c\in M$.
For a compact edge with vertices $v^i$  and $v^j$ let $d^{ij}:=v^j-v^i$ be the oriented edge. 
From the elements appearing in the following definition we will get later
the explicit generators of $\ktT$.

\begin{definition}\label{def 53}
  Let $c\in M$. Assume that $\lan c, d^{ij}\ran\geq 0$. Then we define
\[
\kLp(c, d^{ij}):= \lan c,d^{ij}\ran \,t_{ij}+\{\lan
c,v^j\ran\}s_{v^j}-\{\lan c,v^i\ran\}s_{v^i}.
\]
Moreover, we set $\kLp(c, -d^{ij}):=\kLp(c, d^{ij})$. In particular, the $t_{ij}$-coefficient is always non-negative.
\end{definition}
Note that in the previous definition we do not restrict only to $c\in M\cap \tail(\kP)^\vee$,
but allow any $c\in M$.  

\begin{remark}\label{rem minus}
For $\lan c,d^{ij}\ran=\lan c,v^j-v^i\ran\geq 0$ the following holds: 
\[
  \kLp(-c, d^{ij})=\kLp(-c, -d^{ij})=\kLp(-c,v^i-v^j)=\lan c,d^{ij}\ran \,t_{ij}+\{\lan -c,v^i\ran\}s_{v^i}-\{\lan -c,v^j\ran\}s_{v^j}.
\]
In particular, 
$\kLp(-c, d^{ij})= \kLp(c, d^{ij})-s_{v^j}+s_{v^i}$ unless
$\lan c,v^i\ran, \lan c,v^j\ran\in \Z$. If one of these is integral, then the corresponding $s_{v^i}$ or $s_{v^j}$ has to be omitted  in the previous relation. 
\end{remark}

\begin{remark}\label{rem 53}
One should compare the previous definition with that
of the elements $\kL_{ij}\in \kMTZD(\kP)\otimes_{\ZZ}\kQuotD$ of
Subsection~\ref{ssec:theLatticeStructureInT(P)}. Indeed, for a given
$c\in \tail(\kP)^\vee\cap M$ being non-negative on $d^{ij}$, we have
\[
-\kL_{ij}(c)= \lan c,d^{ij}\ran \,t_{ij}
-\lan c,v^j\ran s_{v^j}+\lan c,v^i\ran s_{v^i},
\]
i.e.\ this differs from $\kLp(c, d^{ij})$ by the integral
$\roundup{\lan c,v^j\ran}s_{v^j}-\roundup{\lan c,v^i\ran}s_{v^i}$.
\end{remark}

\begin{lemma}\label{problematic lemma}
If $d=v^j-v^i$ is an oriented edge of $\kP$, and if $\kc\in M$ such that 
$\braket{\kc,d}=0$, then $\kLp(c,d)=0$.
\end{lemma}

\begin{proof}
If $g_d=1$, then $\lan c,v^j\ran=\lan c,v^i\ran=\lan c,w\ran\in \ZZ$, where $w$ is a lattice point lying on the line passing through $v^i$ and $v^j$ ( $w$ exists, since $g_d=1$). 
If $g_d\geq 2$, then $s_i=s_j$ by definition. Together with
$\lan c,v^j\ran=\lan c,v^i\ran$ this shows the claim. 
\end{proof}

Each path along compact edges determines an element of
$\Z^\kPr$, whose entries count  how often and from which direction we passed through an edge ($\kPr$ is the total number of compact edges).
While this element does not suffice to recover the original path completely, 
we will, nevertheless, call it a path, too.

\begin{definition}\label{def path}
  For $a, \kc\in \tail(\kP)^\vee$ we define (as in \cite{alt}) the following paths on the $1$-skeleton of $\kP$:
\begin{eqnarray*}
  \un{\lam}(a)&:=& [\text{some path } v_*\leadsto v(a)] = [\lam_1(a),\dots,\lam_r(a)]\in\Z^\kPr,\text{~~and}\\
  \un{\mu}^\kc(a)&:=&[\text{some path } v(a)\leadsto v(\kc)\text{ such that }\mu^c_{i}(a)\braket{\kc, d^{i}}\leq 0~\forall~d^{i}] = [\mu^c_1(a),\dots,\mu^c_r(a)]\in\Z^\kPr.
\end{eqnarray*}
Moreover, we define the path $\un{\lam}^c(a):=\un{\lam}(a)+\un{\mu}^c(a)$, 
which is a special path from $v_*$ to $v(c)$ that depends on $a$.
\end{definition}

\begin{remark}
Note that $\lam_i(a)$ or $\mu_i^\kc(a)$ are not uniquely defined. If $d_1,...,d_k$ are oriented edges going from $v_*$ to $v(a)$, i.e.\ it holds that $v(a)=v_*+\sum_{i=1}^kd_i$, then we can choose $\lam_i(a)=1$ for $i=1,...,k$ and $\lam_i(a)=0$ for other $i$. 
\end{remark}

The following lemma is crucial in connecting the generic generators $\kteta(\kc_1,\kc_2)$ of $\ktT$ from Defintion~\ref{def:upperSemigroups},
with the specific elements $\kLp(c,d)$, which will provide an explicit finite set of generators of $\ktT$ (cf. Proposition~\ref{pro 2dim fin gen}).
\begin{lemma}\label{key lemma}
Let $c_1,c_2\in \tail(\kP)^\vee\cap M$ and let $c=c_1+c_2$. For $j=1,2$ we write $\un{\mu}^j(c):=\un{\mu}^{c_j}(c)$ and $\un{\lam}^j(c):=\un{\lam}^{c_j}(c)$. It holds that
\[
\textstyle\ktetaZ(c_1,c_2)= \sum_{j=1}^2 \{\keta(c_j)\}\cdot s_{v(c_j)}
-\{\keta(c)\}\cdot s_{v(c)}
-\sum_{j,\nu} \mu_\nu^j(c)\,\lan c_j,d^\nu\ran\, t_\nu.
\]
\end{lemma}

\begin{proof}
We pick the path $\un{\lam}^j(c)$ from $v_*$ to $v(c_j)$ and compute
\begin{eqnarray*}
\ktetaZ(c_1,c_2)&=&\textstyle\phantom{-}\sum_{j=1}^2\big(\ketaZ(c_j)-\keta(c_j)\big)\cdot s_{v(c_j)}-\big(\ketaZ(c)-\keta(c)\big)\cdot s_{v(c)}-\\
&&\textstyle-\sum_\nu\Big(\sum_j\lam^j_\nu(c)\lan c_j,d^\nu\ran-\lam_\nu(c)\lan c,d^\nu \ran\Big)t_\nu\\
&=&\textstyle\phantom{-}\sum_{j=1}^2\Big(\big(\ketaZ(c_j)-\keta(c_j)\big)\cdot s_{v(c_j)}-\sum_\nu\mu_\nu^j(c)\lan c_j,d^\nu\ran t_\nu\Big)-\big(\ketaZ(c)-\keta(c)\big)\cdot s_{v(c)}.
\end{eqnarray*}
\end{proof}

\begin{lemma}\label{lem vj}
For each vertex $v^j$ and $m\in M$ there exist $m_j\in \tail(\kP)^\vee\cap M$ such that 
$v(m+m_j)=v^j$ and $\lan v,m_j\ran\in \ZZ\,$ for all vertices $v\in P$, even for those not in $N$. 
\end{lemma}
\begin{proof}
Let us take $\widetilde{m}_j\in \tail(\kP)^\vee\cap M$ such that 
$\lan v_j,\widetilde{m}_j\ran<\lan v,\widetilde{m}_j\ran$ 
for all other vertices $v\ne v_j$. Then we take $m_j:=k\widetilde{m}_j$ such that $k$ is big enough that $v(m+m_j)=v^j$ and that additionally 
$\lan v,m_j\ran\in \ZZ$ for all vertices $v\in P$. 
Since the vertices of $P$ have rational coordinates, such an $m_j$ exists.
\end{proof}

\begin{lemma}\label{lem pom lem kon}
It holds that $\kLp(c,d^{ij})\in \ktT$.
\end{lemma}
\begin{proof}
Without loss of generality we assume that $\lan c,d^{ij}\ran\geq 0$. 
With the same argument as in Lemma~\ref{lem vj} there exists an element
$\widetilde{c}_{j}\in \tail(\kP)^\vee\cap M$
which is  perpendicular to the edge $d^{ij}$,
and such that   $\widetilde{c}_{j}-c\in \tail(\kP)^\vee\cap M$ with
$v(\widetilde{c}_{j})=v(\widetilde{c}_{j}-c)=v^j$
and $\lan \widetilde{c}_{j},v^i\ran=\lan \widetilde{c}_{j},v^j\ran\in \ZZ$.
Note that, by assumption we have $\lan -c,v^j\ran\leq \lan -c,v^i\ran$.
We fix such a $\widetilde{c}_{j}$. 
Again by Lemma \ref{lem vj} there exists also $\widetilde{c}_i\in \tail(\kP)^\vee\cap M$
such that $v(\widetilde{c}_{i})=v(\widetilde{c}_i+\widetilde{c}_{j}-c)=v^i$ and $\lan \widetilde{c}_{i},v^i\ran,$ $\lan \widetilde{c}_{i},v^j\ran\in \ZZ$.
Thus 
\begin{equation}\label{eq kteta}
\{\keta(-c+\widetilde{c}_j)\}=\{\lan c,v^j\ran\},\quad\{\keta(\widetilde{c}_i)\}=0,\quad\text{and}\quad\{\keta(-c+\widetilde{c}_i+\widetilde{c}_{j}\}=\{\lan c,v^i\ran\}.
\end{equation} 
 Lemma \ref{key lemma} gives us 
\[
  \ktetaZ(-c+\widetilde{c}_j,\widetilde{c}_i)=\{\keta(-c+\widetilde{c}_j)\}\cdot s_{v^j}+\{\keta(\widetilde{c}_i)\}\cdot s_{v^i}
  -\{\keta(-c+\widetilde{c}_i+\widetilde{c}_j)\}\cdot s_{v^i}- \,\braket{ -c+\widetilde{c}_j,d^{ij}}\, t_{ij}.
\]
Since $\lan d^{ij},\widetilde{c}_j\ran=0$, it follows by \eqref{eq kteta} 
that $\kLp(c,d^{ij})=\ktetaZ(-c+\widetilde{c}_j,\widetilde{c}_i)\in\ktT$.
\end{proof}

\begin{example}\label{ex dij}
Let $\kP=[v^1,v^2] = [-\frac{a_1}{b_1},\frac{a_2}{b_2}]\subset \RR$ with $a_i,b_i>0$ and such that 
$\frac{a_1}{b_1},\frac{a_2}{b_2}\in\QQ\setminus \ZZ$. 
We denote by $m=\operatorname{lcm}(b_1,b_2)$
and set $\vast:=v^1$. So we have $\kQuot=\ZZ$.
For $k\in \ZZ_{>0}$ we have by definition
\begin{equation}\label{eq ktetak}
\ktetaZ(k)=\Big\lceil\frac{a_1k}{b_1}\Big\rceil s_1
	\hspace{1em}\mbox{and}\hspace{1em}
\ktetaZ(-k)=\frac{-a_1k}{b_1}s_1+k\Big(\frac{a_1}{b_1}
	+\frac{a_2}{b_2}\Big)t+\Big\{\frac{a_2k}{b_2}\Big\}s_2.
\end{equation}
Thus, denoting by $\ell(\kP) = v^2-v^1$ the length of $\kP$, we compute for $k\in \{1,\dots,m\}$ that
\begin{eqnarray}
  \label{eq kteta new}
  \ktetaZ(-k,m)&=&k\cdot\ell(\kP)t_{12}+\Big\{\frac{ka_2}{b_2}\Big\}s_2-\Big\{\frac{-ka_1}{b_1}\Big\}s_1,\\
  \nonumber
  \ktetaZ(-m,k)&=&k\cdot\ell(\kP)t_{12}+\Big\{\frac{ka_1}{b_1}\Big\}s_1-\Big\{\frac{-ka_2}{b_2}\Big\}s_2,
\end{eqnarray}
and thus we see that
\begin{equation}\label{eq third t}
\kLp(k,v^2-v^1)=\left\{
\begin{array}{ll}
\ktetaZ(-k,m)& \text{ if }k\in \{1,2,\dots,m\}\\
\ktetaZ(-m,-k)& \text{ if }k\in \{-1,-2,\dots,-m\}.
\end{array}
\right.
\end{equation}
\end{example}

The next lemma follows immediately from the definitions.

\begin{lemma}\label{lem connec t path}
  Let $v^0,v^1\dots,v^k$ be a path along the compact edges of $P$ from vertex $v^0$ to vertex $v^k$.
  For $\kc\in  M$ let 
\[
  \delta_i(c):=\left\{
\begin{array}{ll}
-1,&\text{ if } \lan  c,v^i-v^{i-1}\ran\leq 0,\\
\phantom{-}1,&\text{ if } \lan  c,v^i-v^{i-1}\ran> 0.
\end{array}
\right.
\]
We then have
\[
  \textstyle
  \sum_{i=1}^k\delta_i(c)\,\kLp(c,v^i-v^{i-1})=
	-\{\lan c,v^0\ran\}s_{v^0}+\sum_{i=1}^k\lan c,v^i-v^{i-1}\ran 
	t_{i-1,i}+\{\lan c,v^k\ran\}s_{v^k}.
\]
\end{lemma}

\begin{proposition}\label{prop generators t}
The following set generates $\ktT$:
\[
  \{\kLp(c,d^{ij})\kst d^{ij}\in \cedges(P),~c\in \tail(\kP)^\vee\cap M\}
  \cup
  \{s_v \kst v\in \Vrtx(P)\}.
\]
\end{proposition}
\begin{proof}
Let $c_1,c_2\in \tail(\kP)^\vee\cap M$ and let $c=c_1+c_2$.
As in Lemma \ref{lem connec t path} we compute that 
\[
  \textstyle
  \sum_{\nu=1}^\kPr\kLp(-c_j,\mu_\nu^j(c)d^\nu)=-\{\lan -c_j,v(c)\ran\}s_{v(c)}+\{\lan -c_j,v(c_j)\ran\}s_{v(c_j)}-\sum_{\nu=1}^\kPr\mu_\nu^j(c)\lan d^\nu,c_j\ran t_\nu,
\]
where $\nu$ runs through all the edge indices $1,\dots,\kPr$.
Thus using Lemma \ref{key lemma} we see that 
\[
  \textstyle
  \ktetaZ(c_1,c_2)-\sum_{j=1}^2\sum_{\nu=1}^\kPr\kLp(-c_j,\mu^j_\nu(c)d^\nu)=-\{\eta(c)\}s_{v(c)}+\sum_{j=1}^2\{\lan -c_j,v(c)\ran\}s_{v(c)}.
\]
Since $\{\eta(\kc)\}=\{-\braket{\kc_1+\kc_2,v(\kc)}\}$ by definition,
we see by Lemma \ref{lem trivial} that 
\begin{equation}\label{eq finite g}
  \textstyle
\ktetaZ(c_1,c_2)=\sum_{j=1}^2\sum_{\nu=1}^\kPr\kLp(c_j,\mu^j_\nu(c)d^\nu)+ns_{v(c)},
\end{equation}
where $n$ is either $0$ or $1$, and both actually do appear.
From this description we can easily see that $s_v\in \ktT$ for $v\in \Vrtx(P)$: take $c_1=c_2$ with $v(c_1)=v$ and such that we get $n=1$ above.
Since  $\kLp(c,d^{ij})\in \ktT$ by Lemma \ref{lem pom lem kon}, the equation \eqref{eq finite g} concludes the proof.
\end{proof}

The proof of the  next proposition gives an explicit finite set of generators of $\ktT$.
We will use this for the proof of versality.
Finite generation was also proven in \cite[Proposition 7.7]{a}  with  different methods.

Let $d=w-v$ be an oriented edge  and let 
\[
k:=\min\{ |\lan c,d\ran| \,;\, c\in M, \lan c,d\ran\ne0\}.
\]
Let $c_1\in  M$ be such that $\lan c_1,d\ran=k$.
We define $m_1$ to  be the minimal natural number such that 
$m_1\lan c_1,v\ran,m_1\lan c_1,w\ran\in \ZZ$. 

\begin{proposition}\label{pro 2dim fin gen}
The set
\begin{equation}
  \label{eq:finiteEdgeGens}
  \big\{s_v, s_w\big\} \cup \big\{~\kLp(k_1\cdot c_1,d)\kst k_1= \pm 1, \dots, \pm m_1\big\} 
\end{equation}
generates $\span_\NN\{\kLp(c,d) \kst c\in M\}\subset \ktT$. 
\end{proposition}

\begin{proof}

We choose an arbitrary element $c\in \tail(\kP)^\vee\cap M$ 
and write $c=r_1c_1+c_2$ for some $r_1\in \ZZ$ and 
$c_2\in M$ such that $\lan c_2,d\ran=0$.
Let $\bar{r}_1\in \{1,\dots,m_1\}$ be such that $\bar{r}_1+n_1m_1=r_1$ for some $n_1\in \NN$. 
Without loss of generality  we assume that $\lan c,d\ran,\lan c_1,d\ran\geq 0$.
We obtain that
\begin{equation}\label{eq 2 dim is 1}
\kLp(c,d)-n_1\kLp(m_1c_1,d)-\kLp(\bar{r}_1c_1,d)=\Big(\{\lan c,w\ran\}-\{\lan \bar{r}_1c_1,w\ran\}\Big )s_w-\Big (\{\lan c,v\ran\}-\{\lan \bar{r}_1c_1,v\ran\}\Big )s_v.
\end{equation}
If $g_d=1$, then as in Lemma 
\ref{problematic lemma} we have
\[
  \lan c_2,v\ran=\lan c_2,w\ran=\lan c_2,n\ran\in \ZZ,
\]
where $n\in N$. Thus $\{\lan c,w\ran\}=\{\lan \bar{r}_1c_1,w\ran\}$ and the same for $v$, which proves the claim.\\
\noindent
If $g_d\geq 2$, then $s_v=s_w$  and thus the equation \eqref{eq 2 dim is 1} is equal to $ns_v=ns_w$ for some $n\in \NN$, from which the claim follows.
\end{proof}

\begin{corollary}
The monoid $\ktT$ is finitely generated.
\end{corollary}
\begin{example}\label{ex klaus ex}
Let us consider the one-dimensional polyhedron $\kP=[v,w]\subset\RR$, with $v=-\frac{1}{2}$ and $w=\frac{1}{2}$.
Embedding $\kP$ in height one in $\RR^2$ and dualizing produces the cone
\[
  \sigma\dual=\span_{\RR_{\geqslant0}}\{(-1,2),\;(1,2)\}\subseteq\RR^2.
\]
So the semigroup is
\(
\kS= \sigma\dual\cap\ZZ^2.
\)
The Hilbert basis, i.e.\ the set of minimal generators of $\kS$,
equals
\begin{equation}\label{eq hilb bas el ex}
\{(-2,1),\, (-1,1),\, (0,1),\, (1,1),\, (2,1)\}.
\end{equation}
Since $\kP$ is free from \se\ half open edges, we obtain
$\kMT(\kP)=\RR^3$ with coordinates $(t,s_1,s_2)$.
The oriented edge is  $d=w-v =1 $ and we claim that 
\begin{equation}\label{eq exam eq}
s_1,~~s_2~~,\kLp(1,d)=t+\frac{1}{2}s_2-\frac{1}{2}s_1,~~\kLp(-1,d)=t-\frac{1}{2}s_2+\frac{1}{2}s_1
\end{equation}
is the minimal generating system of $\ktT$.  
Besides the elements  given in \eqref{eq exam eq},
according to (\ref{eq:finiteEdgeGens}) from the proof of Proposition~\ref{pro 2dim fin gen},
we should also take $\kLp(2,d)$ and $\kLp(-2,d)$ as generators.
However, we have  that 
 \[
   \kLp(2,d)=\kLp(-2,d)=2t=\kLp(1,d)+\kLp(-1,d),
 \]
 which concludes our claim. So the generating set presented in~(\ref{eq:finiteEdgeGens}) is finite, but not necessarily minimal.
\end{example}

\begin{example}
  Let us revisit the polyhedron
  $\kP=\conv\big\{(-\frac{1}{6},\frac{1}{2}),\; (\frac{2}{3},\frac{1}{2})\big\}\subset \RR^2$
  from Example~\ref{ex-CQS}.\ref{item:CQS2} In this case $s:=s_1=s_2$ and thus we get the finitely generated semigroup
  \[
\ktT = \span_{\NN}\big\{    s,~~ \frac{5}{6}t+\frac{1}{6}s,~~ \frac{10}{6}t-\frac{4}{6}s\big\}.
  \]
\end{example}

So far we have two generating sets for the semigroup $\ktT$:
the original one $\{\ktetaZ(\kc_1,\kc_2)\}$ from Subsection~\ref{subSectMainMon},
 and the more recent one  $\{\kLp(c,d),\,s_v\}$ from Proposition~\ref{prop generators t}.
The latter are rather local gadgets; they just deal with one  compact edge
$d=w-v$. The following sub-monoid reflects this.
\begin{definition}\label{def submon}
For a compact edge $d=[v,w]$ of $P$ we define the sub-monoid $\ktT_d\subset \ktT$ as
\[
  \ktT_d:=\span_\NN\{s_v,s_w,\kLp(c,d) \kst c\in M\}.
\]
\end{definition}

We are going to discuss the degree of $\kLp(c,d)\in \ktT$ now. 
Assume, for the following that $\langle c,d\rangle\geq 0$, i.e.\ that 
$\langle c,w\rangle \geq \langle c,v\rangle$,
or even, because of Lemma~\ref{problematic lemma}, 
$\langle c,w\rangle > \langle c,v\rangle$.
While it is clear that the degree of $\kLp(c,d)$ equals
\[ 
\big\lceil\lan c,w\ran\big\rceil-\big\lceil\lan c,v\ran\big\rceil
\geq 0,
\]
we will provide a different characterization. For this, we will generalize the
notion of \se\ edges from Definition~\ref{def-shortEdge}
in Subsection~\ref{newAmbient}.

\begin{definition}
\label{def-kshort}
We call $d=[v,w)$ a {\em$k$-\se} (half open) edge if
\[
\;\#\{g_d\cdot [v,w)\cap\kQuotD\}< (k+1)\cdot g_d. 
\]
We call $d=[v,w]$ a {\em$k$-\se} edge if both half open edges $[v,w)$ and $[w,v)$ are {\em$k$-\se}.
In particular, $0$-\se ness means the old plain \se ness.
\end{definition}

\begin{remark}\label{rem klaus remark}
There is a quite subtle relationship between the notion of $k$-shortness
and the true lattice length $\ell:=\ell(d)\in\Q_{\geq 0}$ of an edge $d$.
We have the following implications: 
\[
\textstyle
\Big[\ell\leq (k+1)-\frac{1}{g_d}\Big]
\then
\mbox{ $d$ is $k$-short }
\then
\Big[\ell<k+1\Big].
\]
These two implications are not inverse to each other; the worst case appears
for $g=1$. There, the first expressions just means $[\leq k]$.
\label{rem-kshort}
\end{remark}

Recall the degree map $\pi$ from Definition \ref{def:upperSemigroups}.
\begin{proposition}
\label{prop-degtt-kshort}
Let $d$ be a $k$-\se\ compact edge of $\kP$ which is not
$(k-1)$-\se. Then, 
\[
\min\big\{\kpi\big(\kLp(c,d)\big)\kst \kc\in \kQuot
\mbox{\rm\ with } \langle c,w\rangle \ne \langle c,v\rangle \big\} = k.
\]
\end{proposition}

\begin{proof}
Assume first that $g_d=1$ and denote by 
$v^1,v^2,\ldots,v^\ell$ the sequence of lattice points in the half open 
interval $[v,w)$ with increasing $\kc$-value. 
Then, the assumption means $\ell=k$.
Moreover, denote by 
$v^0$ and $v^{\ell+1}$ the adjacent lattice points, 
hence located outside $[v,w)$.
Then, we have
$\big\lceil\lan c,v\ran\big\rceil\leq \lan c,v^1\ran$ 
and
$\lan c,v^\ell\ran+1 \leq \big\lceil\lan c,w\ran\big\rceil 
\leq \lan c,v^{\ell+1}\ran$.
This implies 
\[
\lceil\lan c,w\ran\big\rceil-\lceil\lan c,v\ran\big\rceil
\geq
\lan c,v^{\ell}\ran +1- \lan c,v^1\ran \geq \ell=k.
\]
On the other hand, let $\kc$ be a special element
of $\kQuot$ such that
$\lan c, v^{i+1}-v^i\ran=1$. Then all the inequalities
in the previous three lines turn into equalities.
\\[1ex]
Let us turn to the case of $g:=g_d\geq 2$.
Again, we name the lattice points $v^1,v^2,\ldots,v^\ell$,
but now inside the half open interval $[gv,gw)$;
the assumption of the proposition means $k\cdot g\leq\ell< (k+1)\cdot g$.
We denote by $g^*$ the first index $i$ such that
$g|\lan c,v^i\ran$. This relation remains valid for all
$i\in (g^*+g\Z)$ among $\{1,\ldots,\ell\}$,
i.e.\ for $i=g^*+\nu g$ with $\nu=0,\ldots,\nu^*:=
\lfloor\frac{\ell-g^*}{g}\rfloor$. 
Now, similarly to the $g=1$ case, we obtain
\[
\textstyle
\big\lceil\lan c,v\ran\big\rceil =
\big\lceil\frac{1}{g}\lan c, gv\ran\big\rceil
\leq \frac{1}{g}\lan c,v^{g^*}\ran
\]
and
\[
\textstyle
\frac{1}{g}\lan c,v^{g^*+\nu^*g}\ran+1 \leq \big\lceil\frac{1}{g}\lan c,gw\ran\big\rceil
\leq \frac{1}{g}\lan c,v^{\ell+1}\ran.
\]
This implies
\[
\textstyle
\lceil\lan c,w\ran\big\rceil-\lceil\lan c,v\ran\big\rceil
\geq
\frac{1}{g}\lan c,v^{g^*+\nu^*g}\ran+1 
- \frac{1}{g}\lan c,v^{g^*}\ran
\geq \nu^*+1 = \lfloor\frac{\ell+g-g^*}{g}\rfloor\geq
\lfloor\frac{\ell}{g}\rfloor=k.
\]
To show that this minimal value can be achieved, we choose again $\kc$
in such a way that $\lan c, v^{i+1}-v^i\ran=1$. Similarly to the first case,
this yields always equality signs until
$\lceil\lan c,w\ran\big\rceil-\lceil\lan c,v\ran\big\rceil
=\lfloor\frac{\ell+g-g^*}{g}\rfloor$. However,
since we may adjust our $\kc$ such that it leads to $g^*=g$, the claim is
proven.
\end{proof}

\begin{corollary}\label{cor ker not 0}
If $\kLp(c,d^{ij})\ne 0$
then the degree of $\kLp(c,d^{ij})$ is strictly bigger than $0$. In particular the kernel of the map $\pi_T=\pi:\ktT\to T=\NN$ is $0$.
\end{corollary}

\begin{proof}
We already know that the degree $\kpi\big(\kLp(c,d^{ij})\big)$ is
non-negative. Moreover, by Definition~\ref{def-kshort}, there is a unique
$k\in\NN$ such that the open half edge $d^{ij}$ is precisely $k$-short,
i.e.\ not $(k-1)$-short. Then, Proposition~\ref{prop-degtt-kshort} implies
that the degree is at least $k$, and it remains to treat the case $k=0$.
\\[1ex]
However, if $d^{ij}$ is $0$-short, i.e.\ short, then we know
that $t_{ij}=s_i$ which already solves the case $g_d\geq 2$, since we have
the equation $s_i=s_j$ anyway. Indeed, having the equations
$s_i=t_{ij}=s_j$, then the elements
$\kLp(c,d^{ij})$ and $\kpi\big(\kLp(c,d^{ij})\big)$ are essentially equal,
i.e.\ the vanishing of the latter implies that of the former. 
\\[1ex]
Finally, if $g_d=1$, then the shortness of $[v^i,v^j)$ immediately 
implies the shortness of $(v^i,v^j]$, unless $v^j\in N$. However, the latter
means $\{\lan c,v^j\ran\}=0$ and $s_j=0$, and we are done again.
\end{proof}

\section{Free pairs}
\label{sec:Diagrams}

In this section we introduce the notion of free pair.
In Subsection \ref{relFlatness} we connect it with free and flat modules.
The results of Subsection \ref{subsec kts} appear in \cite{a} as well;
here we provide a slightly different perspective based on the results from Section \ref{sec sem ktt}.

\begin{definition}
Let $T\subset S$ be two sharp monoids, i.e. commutative semigroups with identity satisfying $S\cap (-S)=\{0\}$.
The \emph{boundary of $S$ relative to $T$} is defined as
\[
\partial_{T}{S} = \{ s\in S~:~(s-T)\cap S = \{s\}\}.
\]
We say that $T\subset S$ form a \emph{free pair} $(T,S)$ if  
the addition map $\operatorname{a}:(\partial_{T}{S})\times T\to S$ is bijective.
\end{definition}
For any free pair, we write the unique decomposition of every element $s\in S$ as 
\[
s = \bo(s) + \la(s)
\hspace{1em}\mbox{with}\hspace{1em}
\bo(s)\in\partial_{T}{S} 
\hspace{0.5em}\mbox{and}\hspace{0.5em}
\la(s)\in S.
\]

\begin{example}
  When $\kT\subseteq \kS$ is the pair of semigroups associated to a rational polyhedron introduced in Section \ref{sec setup},
  we have by \cite[Proposition 2.10 and Remark 5.3]{a} that the pair $(\kT,\kS)$ is a free pair with
\begin{equation}\label{eq partialst}
\partial_TS=\{[\kc,\ketaZ(\kc)] \kst  c\in\tail(P)\dual\cap M\}.
\end{equation}
\end{example}

\subsection{The relation to free modules}
\label{relFlatness}
Let $k$ be any field. Then, the inclusion $\kiota:T\hookrightarrow S$
gives rise to an embedding of semigroup algebras $k[T]\subseteq k[S]$.

\begin{proposition}
\label{prop-straightFlat}
Assume that the addition map $\ksum:\Rand{T}{S}\times T\to S$ is
surjective. Then the pair
$(T,S)$ is free if and only if $k[S]$ is a  
free $k[T]$-algebra, and this holds 
if and only if $k[S]$ is flat over $k[T]$.
\end{proposition}

\begin{proof}
If $(\kT,\kS)$ is a free pair, then the bijection
$\ksum:\Rand{\kT}{\kS}\times\kT\stackrel{_\sim}{\longrightarrow}\kS$ provides an
isomorphism of $k[\kT]$-modules
$\bigoplus_{s\in\Rand{\kT}{\kS}}\,k[\kT]\cdot \chi^s
\stackrel{_\sim}{\longrightarrow}k[\kS]$, i.e.\ $k[\kS]$ is a free $k[\kT]$-module.
\\[1ex]
On the other hand, if $s,s'\in\Rand{\kT}{\kS}$ and $t,t'\in\kT$ with
$s+t'=s'+t$ and $s\neq s'$, 
then we consider the exact sequence of $k[\kT]$-modules
\[
\xymatrix@C4em{
\bigoplus_{i\in I}\, k[\kT]\cdot e_i \ar[r]^-{\sum_i(t_i,t_i')} & 
k[\kT] \oplus k[\kT] \ar[r]^-{\left(
\begin{array}{@{}r@{}}\ks t'\\[-0.5ex] \ks -t\end{array}
\right)} & k[\kT]
}
\]
where $I$ parametrizes a generating set 
$\{(\chi^{t_i},\chi^{t_i'}) 
\}$ of 
$\ker\!\left(\begin{array}{@{}r@{}} \chi^{t'}\\[-0.0ex] 
-\chi^t\end{array}\right)$,
i.e.\ it exhibits the minimal pairs $(t_i,t_i')\in\kT^2$ satisfying 
$t_i+t'=t_i'+t$.
Tensorizing with $\otimes_{k[\kT]}k[\kS]$ replaces
$k[\kT]$ with $k[\kS]$ in the above sequence, and we obtain the new element
\[
(\chi^{s},\chi^{s'})\in \ker\!\left(\begin{array}{@{}r@{}} \chi^{t'}\\[-0.0ex] 
-\chi^t\end{array}\right)\otimes\id_{k[\kS]}.
\]
However, this element cannot be in the image of the first map
$\sum_{i\in I} (t_i,t_i')\otimes\id_{k[\kS]}$. Otherwise, there
is an element $s''\in \kS$ such that 
$(t_i+s'',\,t_i'+s'')=(s,s')$ for some $i$. But then, the
defining property of $\Rand{\kT}{\kS}$ would imply that $t_i=t_i'=0$ and
$s=s''=s'$. Hence, $k[\kS]$ is not flat over $k[\kT]$.
\end{proof}

\begin{remark}
\label{rem-straightFlat}
In \cite[Section 11]{MS05} it was shown by cohomological methods
that for so-called affine 
semigroups $T$, i.e.\ for those being subsemigroups of some
$\ZZ^n$, the $\ZZ^n$-graded flat $k[T]$-modules are direct sums
of degree shifts of localizations of $k[T]$. This fits well to the
consequence of Proposition~\ref{prop-straightFlat} stating that
$k[S]$ is flat over $k[T]$ if and only if it is free
(with basis $\Rand{T}{S}$).
\end{remark}

\subsection{The monoid $\ktS$}\label{subsec kts}
Now we will also start analysing the monoid $\ktS$, from Definition \ref{def:upperSemigroups}. We will show that $(\ktT,\ktS)$ is a free pair (see Corollary \ref{cor free par}) from which it follows that $k[\ktS]$ is a free $k[\ktT]$-module by Proposition \ref{prop-straightFlat}.
Recall the notation $T=\NN$, $S=\cone(\kP)^\vee\cap M$ and recall the two maps $\pi_T:\ktT\to\kT$ and $\pi_S:\ktS\to\kS$ from Definition~\ref{def:upperSemigroups}.

\begin{lemma}
\label{cor-HilbBasisSuff}
The monoid $\ktS$ decomposes as $\,\ktS=\ktT + \spann_{\NN}\big\{
[\kc_1,\ktetaZ(\kc_1)],\dots,[\kc_r,\ktetaZ(\kc_r)]\big\},
$ where $\{[c_1,\ketaZ(\kc_1)],\dots,[\kc_r,\ketaZ(\kc_r)]\big\}$ is the Hilbert basis of $S$.
\end{lemma}
\begin{proof}
Let $\kc\in \tail(\kP)^\vee\cap\kQuot$. Then 
$[\kc,\ketaZ(\kc)]\in\kS$, i.e.\ we know that 
\begin{equation}\label{eq st prva}
\textstyle
[\kc,\ketaZ(\kc)] = \sum_{i=1}^r\lambda_i\,[\kc_i,\ketaZ(\kc_i)]
\end{equation}
for certain $\lambda_i\in\NN$. 
Every element of $\ktS$ can be written as $\kts=[\kc,\ktetaZ(\kc)]+\ktt$
for some $\ktt\in\ktT$. We will prove that
\begin{equation}\label{eq fin gen of tS}
  \textstyle
[\kc,\ktetaZ(\kc)]=\sum_{i=1}^r\lambda_i\,[c_i,\ktetaZ(c_i)].
\end{equation}
Indeed, from \eqref{eq st prva} we have $c=\sum_{i=1}^k\lambda_ic_i$.
So it is enough to prove that $\kLp:=\ktetaZ(\kc)-\sum_{i=1}^k\ktetaZ(c_i)=0$.
This follows since $\kLp\in \widetilde{T}$ and $\pi_{T}(\kLp)=\ketaZ(\kc)-\sum_{i=1}^k\ketaZ(c_i)$,
which is zero by \eqref{eq st prva}.
Since $\ker(\pi_T)=0$, by Corollary \ref{cor ker not 0}, we indeed have $\kLp=0$.
Thus  equation \eqref{eq fin gen of tS} holds.
\end{proof}

\begin{corollary}\label{cor gen ts}
The monoid $\ktS$ is finitely generated. Its generators are the generators of $\ktT$ and 
\begin{equation}\label{eq eq tilde s}
\{\widetilde{s}_1:=[\kc_1,\te(\kc_1)],\dots,\widetilde{s}_r:=(\kc_r,\te[\kc_r)]\}\subset \Rand{\ktT}{\ktS}.
\end{equation}
\end{corollary}

\begin{corollary}\label{cor free par}
The pair $(\widetilde{T},\widetilde{S})$ is free  and we have an isomorphism $\pi_S:\Rand{\ktT}{\ktS}\stackrel{_\sim}{\longrightarrow}\Rand{\kT}{\kS}$. 
\end{corollary}
\begin{proof}
Using Corollary \ref{cor ker not 0} we can easily check that 
\[
  \Rand{\ktT}{\ktS}= \{[\kc,\ktetaZ(\kc)] \kst  c\in\tail(P)\dual\cap M\}.
\]
We get then the isomorphism $\pi_S:\Rand{\ktT}{\ktS}\stackrel{_\sim}{\too}\Rand{\kT}{\kS}$ using the description of $\Rand{\kT}{\kS}$ in the equation \eqref{eq partialst}.\\[1ex]
To prove that $(\ktT,\ktS)$ is free let us 
assume that $\wt{b}_1+\ktt_1=\wt{b}_2+\ktt_2$, with $\wt{b}_i\in\Rand{\ktT}{\ktS}$ and $\ktt_i\in\ktT$. Applying the map $\pi_S:\ktS\to S$ we obtain
\[
  \kpi_S(\wt{b}_1)+\kpi_S(\ktt_1)=\kpi_S(\wt{b}_2)+\kpi_S(\ktt_2).
\]
We have $\kpi(\ktt_1),\kpi(\ktt_2)\in\kT$ and using the isomorphism on the boundaries we get $\kpi_S(\wt{b}_1),\kpi_S(\wt{b}_2)\in\Rand{\kT}{\kS}$.
Since $(\kT,\kS)$ is a free pair, we have that $\kpi_S(\wt{b}_1)=\kpi_S(\wt{b}_2)$. Again by the isomorphism on the boundaries we obtain $\wt{b}_1=\wt{b}_2$, and thus $\ktt_1=\ktt_2$, so the decomposition is unique.
\end{proof}

\section{Syzygies of the free pair $(\ktT,\ktS)$}\label{sec flat}
\subsection{Binomial equations}
Recall from  \eqref{eg hilbbas} the Hilbert basis $\{s_1,\dots,s_r, r\}$ of $\kS$, and the 
liftings $\widetilde{s}_i\in\ktS$  of the $s_i\in\kS$ from \eqref{eq eq tilde s}. 
Let $\kLp_0,\dots,\kLp_g$ be a set of generators of $\ktT$, and thus from
Corollary \ref{cor gen ts} it follows that 
$\kLp_0,\dots,\kLp_g,\widetilde{s}_1,\dots,\widetilde{s}_r$ generate $\widetilde{S}$.
Let us introduce also the following notation

\begin{eqnarray*}
  k[\kS]&=&k[t,x_1,\dots,x_r]/\cI_\kS,\\
  k[\ktT]&=&k[u_0,\dots,u_g]/\cI_\ktT,\\
  k[\ktS]&=&k[u_0,\dots,u_g,x_1,\dots,x_r]/\cI_\ktS.  
\end{eqnarray*}

\begin{definition}\label{def flat 1}
 For $\bfk=(k_1,\dots,k_r)\in \NN^r$ let $x^\bfk:=\prod_{i=1}^rx_i^{k_i}$, 
  and let 
  \[
    \textstyle
    \setlength\arraycolsep{2pt}
    \begin{array}{rclcrcl}
      \bo(\bfk)&:=&\bo(\sum_{j=1}^rk_j[c_j,\eta_\ZZ(c_j)])\in \partial_TS,&\qquad&\la(\bfk)&:=&\la(\sum_{j=1}^rk_j[c_j,\eta_\ZZ(c_j)])\in T,\\
\rule{0pt}{2em} \tbo(\bfk)&:=&\tbo(\sum_{j=1}^rk_j[c_j,\kteta_\ZZ(c_j)])\in \partial_\ktT\ktS,&&\tla(\bfk)&:=&\tla(\sum_{j=1}^rk_j[c_j,\kteta_\ZZ(c_j)])\in \ktT.
    \end{array}
  \]
\end{definition}
Note that the isomorphism $\Rand{\ktT}{\ktS}\stackrel{_\sim}{\too}\Rand{\kT}{\kS}$ from Corollary \ref{cor free par}
sends  $\tbo(\bfk)$ to $\bo(\bfk)$.
We will identify the two and write $\bo(\bfk)=\tbo(\bfk)$.
Note also that the map $\pi_T$ from Definition \ref{def:upperSemigroups} maps $\tla(\bfk)$ to $\la(\bfk)$.

\begin{definition}\label{def flat 2}
For each element $s\in S$ (resp. $\widetilde{s}\in \ktS$) we fix a representation
$s=a_0R+\sum_{i=1}^ra_is_i$ (resp.
$\widetilde{s}=\sum_{j=0}^gn_j\kLp_j+\sum_{i=1}^ra_i\widetilde{s}_i)$ 
with $\sum_{i=1}^ra_i\widetilde{s}_i=\sum_{i=1}^ra_is_i$ inside
$\Rand{\ktT}{\ktS}\stackrel{_\sim}{\too}\Rand{\kT}{\kS}$.
Define
\[
  \textstyle
  x^s:=t^{a_0}\prod_ix_i^{a_i},\qquad x^{\bo(\widetilde{s})}:=\prod_ix_i^{a_i},\qquad u^{\tla(\widetilde{s})}:=\prod_ju_j^{n_j}.
\]
In particular we can present $\bo(\bfk)=\tbo(\bfk)$ as an element of $\NN^r$.
We define the binomials
\[
  f_\bfk:=x^\bfk - x^{\bo(\bfk)}\,t^{\la(\bfk)},~~~F_\bfk:=x^\bfk -  x^{\bo(\bfk)}\,u^{\tla(\bfk)}.
\]
\end{definition}

\begin{lemma}\label{l:liftingRelations}
The binomials $f_\bfk$ generate the ideal $I_\kS=\ker (\varphi:k[t,x_1,\dots,x_r]\to k[\kS])$
 and the binomials $F_\bfk$ generate the ideal
 \(
 \ker( \wt{\varphi}:k[\ktT][x_1,\dots,x_r] \longrightarrow k[\ktS]).
 \)
\end{lemma}
\begin{proof}
Let us only prove the second statement, the first one follows analogously.
By construction we have $F_\bfk \in \ker(\wt{\varphi})$. 
Since $\ker(\wt{\varphi})$ is $\ktS$-homogeneous, the kernel is spanned by binomials of the form 
\[
  u^\bfa x^\bfk-x^{\bo(\bfa+\bfk)}u^{\tla(\bfa+\bfk)}=u^\bfa F_\bfk,
\]
where $\bfa\in \NN^{r}$, which concludes the proof.
\end{proof}

\subsection{Lifting syzygies}
We start with a general lemma which will turn out useful. 
\begin{lemma}\label{r:projectionOfSum}
For any free pair $(\kT,\kS)$ and for any $w_1,w_2 \in \kS$ we have
\begin{eqnarray*}
  \bo(w_1+w_2)&=&\bo(\bo(w_1)+\bo(w_2)),\\
  \la(w_1+w_2)-\la(w_1)-\la(w_2)&=&\la(\bo(w_1)+\bo(w_2)).
\end{eqnarray*}
\end{lemma}
\begin{proof} To conclude it is enough to apply the unique decomposition of $w_1+w_2$ in the following:
\begin{eqnarray*} 
  \bo(w_1+w_2)+\la(w_1+w_2)&=& w_1 + w_2\\
                           &=& \bo(w_1)+\bo(w_2)+\la(w_1)+\la(w_2)\\
                           &=&\bo(\bo(w_1)+\bo(w_2))+\la(\bo(w_1)+\bo(w_2))+\la(w_1)+\la(w_2).
\end{eqnarray*}
\end{proof}

Let $\cR$ denote the kernel of the map
\[
  \textstyle
  \psi:\bigoplus_{\bfk\in \NN^r}k[t,x_1,\dots,x_r]e_\bfk\xrightarrow{e_\bfk\mapsto f_\bfk} \cI_\kS\subset k[t,x_1,\dots,x_r].
\]
Thus $\cR$ is the module of linear relations between the $f_\bfk$. 
\begin{definition}\label{d:sysygyGenerators}
  For every $\bfa,\bfk\in\NN^r$ we define 
  \[
    R_{\bfa,\bfk} := e_{\bfa+\bfk} - x^\bfa e_\bfk - t^{\la(\bfk)}e_{\bo(\bfk)+\bfa}.
  \]
\end{definition}
To check that $R_{\bfa,\bfk}\in \cR$ we compute:
  \begin{eqnarray*}\label{eq pom com}
    \psi(R_{\bfa,\bfk}) & = & x^{\bfa+\bfk}- x^{\bo(\bfa+\bfk)}t^{\la(\bfa+\bfk)} - \\
    && - x^\bfa\left(  x^\bfk - x^{\bo(\bfk)}t^{\la(\bfk)} \right) - t^{\la(\bfk)}\left( x^{\bo(\bo(\bfk)+\bfa)} - x^{\bo(\bo(\bfk)+\bfa))}t^{\la(b(\bfk)+\bfa)}\right)\\  
 &=&x^{\bo(\bfa+\bo(\bfk))}t^{\la(\bfa+\bo(\bfk))+\la(\bfk)}-x^{\bo(\bfa+\bfk)}t^{\la(\bfa+\bfk)}=0,
  \end{eqnarray*}
where the last equality we obtain by Lemma \ref{r:projectionOfSum}.

\begin{lemma}\label{l:syzygyGenerators}
The module $\cR$ is spanned by $R_{\bfa,\bfk}$, for $\bfa,\bfk\in \NN^r$. 
\end{lemma}
\begin{proof}
  Let $R=\sum g_ie_{\bfk_i}\in \cR$ be a homogeneous relation in $\kS$-degree $w$.
Computing modulo $\langle R_{\bfa,\bfk}\rangle$, we can always replace $x^\bfa e_\bfk$ by $e_{\bfa+\bfk} - t^{\la(\bfk)}e_{\bo(\bfk)+\bfa}$. So we may assume that each $g_i=\alpha_it^{a_i}$. 
Moreover, we can assume that there exists an index $i$ such that $a_i=0$ (otherwise, divide $R$ by the minimal power of $t$). Let $R_0=\sum_{a_i=0}\alpha_it^{0}e_{\bfk_i}$. We have $\psi(R_0)=\sum_{a_i=0}\alpha_i(x^{\bfk_i} - x^{\bo(\bfk_i)}t^{\la(\bfk_i)})=0$,  and furthermore, each $\alpha_ix^{\bfk_i}$ must cancel with some $\alpha_jx^{\bo(\bfk_j)}t^{\la(\bfk_j)}$, so $\la(\bfk_j)=0$ for each $j$ with $a_j=0$. This, together with $S$-homogeneity, implies that $x^{\bo(\bfk_j)}=x^{\bo(w)}$ for all $j$ with $a_j=0$. Thus actually $R_0$ is the empty sum, contradicting the existence of an $a_i=0$ in $R$.
\end{proof}

Let $\widetilde{\cR}$ denote the kernel of the map
\[
  \widetilde{\psi}:\bigoplus_{\bfk\in \NN^r}k[u_0,\dots,u_g,x_1,\dots,x_r]E_\bfk\xrightarrow{E_\bfk\mapsto F_\bfk} \cI_\ktS\subset k[u_0,\dots,u_g,x_1,\dots,x_r].
\]
Thus $\widetilde{\cR}$ is the module of linear relations between $F_\bfk$. 

\begin{definition}\label{d:liftingSyzygies}
  For each $\bfa,\bfk \in \NN^r$ we define the relation among the generators of $\ktS$ given in Lemma~\ref{l:liftingRelations}:
\[ \wtR_{\bfa,\bfk}=E_{\bfa+\bfk} - x^\bfa E_\bfk - u^{\tla(\bfk)}E_{\tbo(\bfk)+\bfa}.\]  
\end{definition}

As we did for $R_{\bfa,\bfk}$ we also compute in this case that 
\begin{equation}\label{eq rel rac}
\widetilde{\psi}(\wtR_{\bfa,\bfk})=x^{\bo(\bfa+\bo(\bfk))}u^{\tla(\bfa+\bo(\bfk))+\tla(\bfk)}-x^{\bo(\bfa+\bfk)}u^{\tla(\bfa+\bfk)},
\end{equation}
which is equal to $0$ in $k[\ktT][x_1,\dots,x_r]$ by Lemma \ref{r:projectionOfSum}. 
In particular, $R_{\bfa,\bfk}$ lifts to $\wtR_{\bfa,\bfk}$.

\subsection{Explicit description of $\tla(\bfk)$}

We write $\bfk=(k_1,\dots,k_r)\in \NN^r$ and $c=\sum_{i=1}^rk_ic_i$,
where the $c_i\in\kQuot$ are the elements appearing in the Hilbert basis of $S$, see \eqref{eg hilbbas}.
 Recall the elements $\tla(\bfk)$ and $\bo(\bfk)$ from Definition \ref{def flat 1}.
\begin{lemma}\label{lem pla lem}
  For all $\bfk\in\NN^r$ we have $\bo(\bfk)=[\kc,\eta_\ZZ(\kc)]$, 
    $\tbo(\bfk)=[\kc,\te(\kc)]$ and
\[\textstyle
  \tla(\bfk)=\te(\bfk):=[0,\sum_{i=1}^rk_i\te(\kc_i)-\te(\kc)].
  \]
\end{lemma}
\begin{proof}
We have 
$\sum_{i=1}^rk_i[c_i,\te(c_i)]=[0,\sum_{i=1}^rk_i\te(c_i)-\te(c)]+[c,\te(c)]$ with $[c,\te(c)]\in \ktS$ and $[0,\sum_{i=1}^rk_i\te(c_i)-\te(c)]\in \ktT$, which concludes the proof.
\end{proof}

By Definition \ref{def flat 2} we treat $\bo(\bfk)$ as an element of $\NN^r$, say $\bo(\bfk)=(b_1,\dots,b_r)\in \NN^r$.
This means that 
\begin{equation}\label{eq boundary}
  \textstyle
\bo(\bfk)=[c,\eta_\ZZ(c)]=\sum_{j}b_j[c_j,\eta_\ZZ(c_j)].
\end{equation}
 Recall the definition of $v(c)$ from Section \ref{sec setup} and the paths $\un{\lam}(a), \un{\mu}^j(a), \un{\lam}^j(a)$ from Definition \ref{def path}.

\begin{lemma}\label{lem pom gen 1}
It holds that $\sum_{j=1}^r\lam_\nu^j(c)b_j\lan c_j,d^\nu\ran=\lam_\nu(c)\lan c,d^\nu\ran$ for each compact edge $d^\nu$.
\end{lemma}
\begin{proof}
 Let $F$ be the face of $P$ where $c$ attains its minimum and let $F_j$ be the face of $P$ where $c_j$ attains its minimum. Then $b_j\ne 0$ only for those $j$ such that $F_j\subset F$, from which the proof easily follows.
\end{proof}

The following description of $\tla(\bfk)$ will be important in Section \ref{sec loop}.
\begin{proposition}\label{prop tildelam}
For $\bfk=(k_1,\dots,k_r)$, $\bo(\bfk)=(b_1,\dots,b_r)$ and $c=\sum_{i=1}^rk_ic_i$ it holds that 
\[
  \textstyle
  \tla(\bfk)=\sum_j(k_j-b_j)\Big(\sum_\nu\delta_{j,\nu}(c)\kLp\big(\lam^j_\nu(c)c_j,d^\nu\big)+\{\lan c_j,v_{*}\ran\}s_{v_{*}}\Big)\in \ktT,
\]
 where 
 \[
   \delta_{j,\nu}(c):=\left\{\begin{array}{ll}
1& \text{ if }\lan c_j,\lam^j_\nu(c) d^\nu\ran>0\\
0&\text{ if }\lan c_j,\lam^j_\nu(c) d^\nu\ran \leq0.
 \end{array}
 \right.
 \]
 \end{proposition}
\begin{proof}
 In the definition of $\tla(\bfk)$ let us pick the path $\un{\lam}^j(c)$ from $v_*$ to $v(c_j)$.
 We compute
 \begin{eqnarray*}
   \tla(\bfk) & = & \textstyle
                 \sum_jk_j\big(\ketaZ(c_j)-\keta(c_j)\big)\cdot s_{v(c_j)}-\big(\ketaZ(c)-\keta(c)\big)\cdot s_{v(c)}+ \\
              &   & \textstyle
                +\sum_\nu\Big(\sum_jk_j\lam^j_\nu(c)\lan c_j,d^\nu\ran-\lam_\nu(c)\lan c,d^\nu \ran\Big)t_\nu          \\
              & = & \textstyle
                 \sum_j(k_j-b_j)\Big(\big(\ketaZ(c_j)-\keta(c_j)\big)\cdot s_{v(c_j)}+\sum_\nu\lam_\nu^j(c)\lan c_j,d^\nu\ran t_\nu\Big),
 \end{eqnarray*}
 where in the last equality we used the equation \eqref{eq boundary} and Lemma \ref{lem pom gen 1}. 
By Lemma \ref{lem connec t path} we see that 
\[
  \textstyle
  \sum_\nu\delta_{j,\nu}\kLp(\lam^j_\nu(c) c_j, d^\nu)=\big(\ketaZ(c_j)-\keta(c_j)\big)\cdot s_{v(c_j)}-\{\lan c_j,v_{*}\ran\}s_{v_{*}}+\sum_\nu\lam_\nu^j(c)\lan d^\nu,c_j\ran t_\nu,
\]
from which we conclude the proof.
\end{proof}
\section{The deformation diagram}\label{sec vers}
\subsection{The free pair $(\ktT,\ktS)$ yields a deformation of a hyperplane section}
\label{defHyperplaneSections}

The injection 
$\kT \hookrightarrow \kS$ 
yields a morphism $R:X=\spec k[\kS] \to\A^1_k$.
Its zero-fiber $Z:=R^{-1}(0)\subseteq X$ equals
$\Spec k[\Rand{\kT}{\kS}]$ where the definition of the
$k$-vector space $k[\Rand{\kT}{\kS}]$ is straightforward, and it becomes a
$k$-algebra via the multiplication law saying that for
$s,s'\in\Rand{\kT}{\kS}$ we set
\[
\chi^s\cdot\chi^{s'}:=\left\{
\begin{array}{ll}
\chi^{s+s'} & \mbox{if } s+s'\in\Rand{\kT}{\kS}, \mbox{ i.e.~if }
\height(s+s')=0
\\[0.5ex]
0 & \mbox{if } \height(s+s')>0.
\end{array}
\right.
\]

\begin{example}
\label{ex-discreteC}
Let us consider Example \ref{ex klaus ex}. 
Here, the equations for $Z\subseteq\A^4_k$ are
$z_i\,z_{-j}=0$ ($i,j=1,2$) and $z_1^2=z_{-1}^2=0$, where $z_i$ is the coordinate corresponding to the Hillbert basis element $(i,1)$ in \eqref{eq hilb bas el ex}. Hence, $Z$ is the union
of two orthogonal double lines.
\end{example}

We have the commutative diagram
\[
  \begin{tikzcd}
    \big[Z=R^{-1}(0)\big] \arrow[r,hook] \arrow[d]
    & 
    X \arrow[r,hook] \arrow[d]{}{R}
    & 
    \Spec k[\ktS] \arrow[d]{}{\wt{R}}
    \\
    0 \arrow[hook,r]
    &
    \A^1_k \arrow[hook,r]
    &
    \Spec k[\ktT].    
  \end{tikzcd}
\]

By Proposition \ref{prop-straightFlat},
all vertical maps are flat, and both squares are Cartesian diagrams.
That is, both $R:X\to\A^1_k$ and
$\wt{R}:\Spec k[\ktS]\to\Spec k[\ktT]$ are deformations of 
$Z=R^{-1}(0)$. 

\subsection{Deformations of $X$ instead of $Z$}\label{sectZX}

From now on we assume that $\ktT$ is generated by degree $1$ elements.
Lemma \ref{lem lepa} below offers  a geometric interpretation of this condition.

There is an alternative possibility to produce a deformation diagram out of
the right hand square of the  diagram in Subsection \ref{defHyperplaneSections}:
\begin{equation}
  \label{eq diag 2}
  \begin{tikzcd}
    Z \arrow[d] \arrow[hook,r]
    &
    X \arrow[hook,r] \ar[d]{}{R}
    & 
    \tX \arrow[hook,r] \arrow[d]{}{\wt{R}}
    &
    \Spec k[\ktS]\arrow[d]{}{\wt{R}}
    \\
    0 \arrow[hook,r]
    &
    \A^1_k \arrow[hook,r] \arrow[d]
    & 
    \baseMfullPre \arrow[hook,r] \arrow[d]{}{\ell}
    & 
    \Spec k[\ktT] \arrow[hook,r] \arrow[Rightarrow,dl,"\mbox{\tiny maximal}"]
    & 
    \A^{\gens+1}_k\ar[d,"\ell"]
    \\
    &
    0 \arrow[hook,r]
    & 
    \baseMkick \arrow[hook,rr]
    &&
    \A^{\gens+1}_k/\Delta.
  \end{tikzcd}
\end{equation}
We described the most important part of the diagram \eqref{eq diag 2} already in Introduction, see diagram \eqref{eq imp imp dia}.
The double arrow between $\spec k[\ktT]$ and $\baseMkick$ is supposed to
indicate that there is a maximal closed 
subscheme $\baseMkick\subseteq\A^\gens_k$ meeting the requirement
$\ell^{-1}(\baseMkick)\subseteq \spec k[\ktT]$. One obtains the ideal 
providing this distinguished maximal $\baseMkick$ as follows: 
write all (binomial) equations $f(u_0,\ldots,u_\gens)$
from the ideal of $\spec k[\ktT]\subseteq\A^{\gens+1}_k$ 
in coordinates
$u_0, T_1,\ldots,T_\gens$ with 
$T_i:=u_0-u_i$ ($i=1,\ldots,\gens$),
such as
\[
\textstyle
f(u_0,\dots,u_g)
=\sum_{l\geq 0} \,f_l(T_1,\ldots,T_\gens)\cdot u_0^l.
\]
Then by definition, the ideal of $\baseMkick$ is generated by the coefficients
$f_l(T_1,\ldots,T_\gens)\in k[T_1,\dots,T_g]$.

Recall the sub-monoids $\ktT_d\subset \ktT$ from Definition \ref{def submon}. 
The main result of this paper is the following.
\begin{theorem}
\label{th main}
Let $X$ be a toric variety from our setup in Section \ref{sec setup}. Assume that
 $\ktT_d$ is generated by degree $1$ elements for all compact edges $d$ of $P$.
Then the maximal $\baseMkick\subseteq\A_k^\gens$ with $\ell^{-1}(\baseMkick)\subseteq
\spec k[\ktT]$ yields a maximal deformation
with prescribed tangent space $T^1_X(-R)\subseteq T^1_X$.
\end{theorem}

\begin{remark}\label{rem cod 2 smooth}
Theorem \ref{th main} 
has been 
shown in \cite{alt} and \cite{budapest} for the special case of $X$ 
lacking singularities in codimension two. 
In the combinatorial language of polytopes this means that all two faces $\lan a^i,a^j\ran$ are smooth, i.e. $a^i$ and $a^j$ are the base of $\lan a^i,a^j\ran\cap N$.
Lemma \ref{lem lepa} will show to what extent Theorem~\ref{th main} is a generalization of this case.

\end{remark}

Clearly the assumption that $\ktT_d$ is generated by degree $1$ elements for all $d$ implies that $\ktT$ is generated by degree $1$ elements by the description of the generators of $\ktT$ in Proposition \ref{prop generators t}. 

\begin{remark}
Note that from Proposition \ref{prop-compTOne} we see that
$\baseMkick$ has indeed $T^1_X(-R)$ as its tangent space.  Moreover, the assumption $\ktT_d$ is generated by degree $1$ elements for all $d$ implies that all edges are $1$-short by Proposition \ref{prop-degtt-kshort}. This implies that $T^1_X(-kR)=0$ for $k\geq 2$ by Proposition \ref{prop-compTOne}.
\end{remark}

\begin{example}\label{ex house}
Let 
$P=\conv\{(0,0),(2,0),(2,1),(1,2),(0,1)\}\subset \RR^2$ be the lattice polygon: 
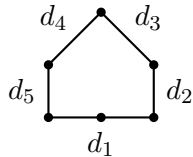
\[
\begin{tikzpicture}[scale=0.7]
\draw[thick,  color=black]  
  (0,0) -- (2,0) -- (2,1) -- (1,2) -- (0,1) -- cycle;
\fill[thick,  color=black]
  (0,0) circle (2.5pt) (1,0) circle (2.5pt) (2,0) circle (2.5pt) (2,1) circle (2.5pt)
  (1,2) circle (2.5pt) (0,1) circle (2.5pt);
\draw[thick,  color=black]
  (1,-0.5) node{$d_1$} (2.5,0.5) node{$d_2$} (1.9,1.9) node{$d_3$}
  (0.1,1.9) node{$d_4$} (-0.5,0.5) node{$d_5$};
\end{tikzpicture}
\]
The generators of $\ktT$ are $2t_1,~t_2,~t_3,~t_4,~t_5$ (they  correspond to the five edges). The closing condition~\eqref{2 face equation} on $2$-faces gives us that $2t_1=t_3+t_4$, which implies that $\ktT$ is generated by degree $1$ elements but on the other hand we see that $\ktT_{d_1}\subset \ktT$ is not generated by degree $1$ elements (it is generated by $2t_1$). 
\end{example}

\begin{remark}
Example \ref{ex house} shows that the condition that $\ktT_d$ are generated by degree $1$ elements is slightly stronger than the condition that $\ktT$ is generated by degree $1$ elements. See also Remark \ref{rem finrem} why we impose this stronger condition.
\end{remark}

\begin{lemma}\label{lem lepa}
For a compact edge $d=[v,w]$ of $P$
assume that one of the following holds:
\begin{enumerate}
\item[a)] $d$ is a short edge;
\item[b)] one vertex of $d$ lies in $N$ and the lattice length of $d$ is strictly smaller than $2$;
\item[c)] $g_d\geq 2$ and $d$ is $1$-short;
\item[d)] $g_d=1$ and d is $1$-short and the lattice length of $d$ is bigger than $1$;
\item[e)] there is an isomorphism of the lattice $N$ that maps the edge $d$ to the edge with vertices $(-\frac{1}{n},0,\dots,0)$ and $(\frac{1}{m},0,\dots,0)$, $n,m\in \NN$ and $n,m\ne 1$.
\end{enumerate}
Then $\ktT_d$ is generated by degree $1$ elements.
\end{lemma}
\begin{proof} 
  Let $c'\in M$ be such that
  \[
    \braket{c',d}=\min\{ |\braket{\kc,d}| \kst \kc\in \kQuot, \braket{\kc,d}\ne0\}.
  \] 
  Note that this value appeared in the proof of Proposition \ref{pro 2dim fin gen}, where we described the generators of $\ktT$.
  
  If a) holds the claim trivially follows.
  
If b) holds we may assume that $w\in N$.
Then we can easily see that the semigroup $\ktT_d$ is generated by two elements, namely $s_v$ and $\kLp(c',d)$. Since the lattice length of $d$ is strictly smaller than $2$ we see that $\kLp(c',d)$ has degree $1$ by Proposition \ref{prop-degtt-kshort}, from which the claim follows.

If c) holds, then $s_v=s_w$ and as in b) we can easily verify that $\ktT_d$ is generated by two elements, namely $s_v=s_w$ and $\kLp(c',d)$. By Proposition \ref{prop-degtt-kshort} we see that the degree of $\kLp(c',d)$ is $1$.

 If d) holds, then we can easily check that $\ktT_d$ is generated by four elements: $s_v$, $s_w$, $\kLp(c',d)$, $\kLp(-c',d)$, which have degree $1$ by Proposition \ref{prop-degtt-kshort}.
 
Let us now assume that only e) holds. It is enough to show that for $d=[v,w]=[-\frac{1}{n},\frac{1}{m}]\subset \RR$ the semigroup $S'=\span_\NN\{s_v,s_w,\kLp(c,d) \kst c\in M\}$  is generated by degree $1$ elements. This is clear by explicit description of the generators described in Proposition \ref{pro 2dim fin gen}: 
let us first consider the case when $m,n\ne 1$. In this case $S'$ is isomorphic to the following monoid (see also Example \ref{13 example} for a geometric picture): 
let the polytope $Q$ be the convex hull of the vertices $(0,0), (0,1), (n-1,0), (m-1,1)$. Let $C(Q)$ be the cone over this polytope, i.e. generated by $(0,0,1), (0,1,1), (n-1,0,1), (m-1,1,1)$. We will show that the monoid  $\ZZ^3\cap C(Q)$ is isomorphic to $S'$. Indeed, the isomorphism is given by 
\[
  s_v\mapsto (0,0,1),~~~~s_{w}\mapsto (0,1,1),~~~~\kLp(a,d)\mapsto (a,1,1),~~~~\kLp(-b,d)\mapsto (b,0,1),
\]
for $a=1,\dots,n-1$ and $b=1,\dots,m-1$.  
From this we conclude the proof.  
\end{proof}

\begin{remark}
  The compact edges of $P$ correspond to the two dimensional cyclic quotient singularities.
  Thus Lemma~\ref{lem lepa} can be phrased using this language as well, see  \cite[Section 2]{alt-kol}.
\end{remark}

\begin{example}
Let $d=[-\frac{2}{3},\frac{1}{4}]\subset \RR$. Then $\ktT_d$ is generated by degree $1$ elements, namely by 
$$s_1,~~~s_2,~~~,\kLp(1,d)=\frac{11}{12}t+\frac{3}{4}s_2-\frac{2}{3}s_1,~~~\kLp(-1,d)=\frac{11}{12}t+\frac{1}{3}s_1-\frac{1}{4}s_2.$$
Thus we see that the list in Lemma \ref{lem lepa} is not exhaustive, i.e. $\ktT_d$ is generated by degree $1$ elements but it does not appear on the list.
\end{example}

\begin{example}
Let us consider $P=[v_1,v_2]=[-\frac{3}{5},\frac{1}{5}]\subset \RR$. Here we have only one edge, which is $1$-short and we will show that $\ktT$ is not generated by degree $1$ elements.
The degree $1$ elements are 
\begin{equation}\label{eq deg 1 elt}
s_1,~~~s_2,~~~\kLp(1,d)=\frac{4}{5}t+\frac{4}{5}s_2-\frac{3}{5}s_1,~~~\kLp(-1,d)=\frac{4}{5}t+\frac{2}{5}s_1-\frac{1}{5}s_2.
\end{equation}
We see that we can not write the element
\[
  \kLp(3,d)=\frac{12}{5}t+\frac{2}{5}s_2-\frac{4}{5}s_1
\]
as a sum of degree 1 elements in \eqref{eq deg 1 elt}, thus $\ktT$ is not generated by degree $1$ elements. We have $g_d=1$ and the lattice length of $d$ is smaller than $1$ thus none of the conditions in Lemma \ref{lem lepa} is satisfied for this example.
\end{example}

\begin{example}
Let us continue with Example \eqref{ex klaus ex}. 
Let us denote $k[\ktT]=k[u_1, u_2, u_A, u_B]$, where the variables correspond to the minimal generating set of $\ktT$, written in \eqref{eq exam eq} (here $A=\kLp(1,d)$ and $B=\kLp(-1,d)$). We only have the following binomial equation
\[
u_A\,u_1-u_B\,u_2=0.
\]
Writing $u_0=u_1$ and $T_i=u_0-u_i$ for $i=2,A,B$, 
turn this equation into
\[
(u_0-T_A)u_0-(u_0-T_B)(u_0-T_2)=0
\]
from which we get
\[
  u_0(-T_A+T_B+T_2)-T_2T_B=0.
\]
The equations of our versal base space are
\[
  T_A=T_2+T_B,~~~T_2T_B=0
\]
and thus
\[
\baseMkick=\Spec k[T_2,T_B]/(T_2T_B),
\]
i.e.\ $\baseMkick$ equals the union of two lines. These two lines correspond to the two Minkowski decompositions 
\[
  P=P_1+P_2=Q_1+Q_2,
\]
where $P_1$ is the vertex $-\frac{1}{2}$,  $P_2=[0,1]$, $Q_1=[-\frac{1}{2},0]$ and $Q_2=[0,\frac{1}{2}]$, see also \cite[Section 9]{a}, where those Minkowski decompositions were called lattice friendly Minkowski decompositions. 
\end{example}

\begin{example}\label{13 example}
Let  $\kP=[v,w]=[-\frac{1}{2},\frac{1}{3}]\subset\RR$ and thus  
 $\sigma=\span_{\RR\geqslant 0} \{(-1,2),\;(1,3)\}\subseteq\RR^2$.
Dualizing we obtain  free embedding of monoids
\[
  \begin{tikzcd}
\kT=\NN\ar[hook,r]& \span_{\RR_{\geqslant 0}} \{(-3,1),\,(2,1)\}\cap\ZZ^2=\kS.    
  \end{tikzcd}
\]
The Hilbert basis of $\kS$, i.e.\ the set of minimal generators, 
equals
\[
\{(-3,1),\, (-2,1),\, (-1,1),\, (0,1),\, (1,1),\, (2,1)\}.
\]
Since $\kP$ is free from \se\ half open edges, we obtain
$\kMT(\kP)=\RR^3$ with coordinates $(t,s_v,s_w)$. We see that
the elements
\[
  A=\kLp(1,d)=\frac{5}{6}t+\frac{2}{3}s_w-\frac{1}{2}s_v,\,\,~~~B_1=\kLp(-1,d)=\frac{5}{6}t+\frac{1}{2}s_v-\frac{1}{3}s_w,\,\,~~~ B_2=\kLp(2,d)=\frac{5}{3}t-\frac{2}{3}s_w,
\]
together with $s_v$ and $s_w$ generate $\ktT$. Thus $\ktT$
is the set of lattice points of the cone over the quadrangle
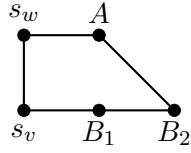
\[
\begin{tikzpicture}[scale=1.0]
\draw[thick,  color=black]  
  (0,0) -- (2,0) -- (1,1) -- (0,1) -- (0,0);
\fill[thick,  color=black]
  (0,0) circle (2.5pt) (2,0) circle (2.5pt) (1,1) circle (2.5pt)
  (0,1) circle (2.5pt) (1,0) circle (2.5pt) ;
\draw[thick,  color=black]
  (0,-0.3) node{$s_v$} (1,1.3) node{$A$} (2,-0.3) node{$B_2$}
  (1,-0.3) node{$B_1$} (0,1.3) node{$s_w$};
\end{tikzpicture}
\]
The generators obey the affine relations
\[
A+s_v=B_1+s_w
\hspace{0.8em}\mbox{and}\hspace{0.8em}
2B_1=B_2+s_v,
\]
which induce the following binomial equations:
\[
u_A\,u_v-u_{B_1}\,u_w
\hspace{0.8em}\mbox{and}\hspace{0.8em}
u_{B_1}^2-u_{B_2}\,u_v.
\]
After writing $u_0=u_{B_1}$ and $T_i=u_0-u_i$ for $i=v,w,A,B$, 
these equations turn into
\[
(u_0-T_A)(u_0-T_v)-u_0(u_0-T_w)=0,~~~
u_0^2-(u_0-T_{B_2})(u_0-T_v)=0.
\]
and thus after some computation we obtain 
\[
  u_0(T_w-T_v-T_A)+T_AT_v=0,~~~u_0(T_v+T_{B_2})-T_{B_2}T_v=0.
\]
We can write
\[
  T_w=T_v+T_A,~~~T_{B_2}=-T_v
\]
and we end up with
\[
\baseMkick=\Spec k[T_v,T_A]/(T_v^2,T_AT_v),
\]
i.e.\ $\baseMkick$ equals the line with an embedded point. This line is corresponding to the Minkowski decomposition $P=[-\frac{1}{2},0]+[0,\frac{1}{3}]$.
\end{example}

\subsection{The Obstruction map}\label{subsec obs}

From  \cite[Section 4]{jong} and \cite[Section 7]{alt} (see also \cite[Section 10]{de jong}) 
 we recall the definition of the obstruction map, which is the main tool for proving Theorem \ref{th main}. 
 As in Subsection \ref{r:projectionOfSum}, let $\cR$ be the module of linear relations between $f_\bfk$,
 which are the generators of $\cI_\kS$.
 The module $\cR$ contains the submodule $\cR_0$ of the so-called Koszul relations.  
 \begin{definition}\label{def t2}
   Let $\kS$ be the monoid defined by $\kP$ and $X=\Spec k[\kS]$. We define
   \[
     T^2_X:=\frac{\Hom(\cR/\cR_0,\,k[\kS])}{\Hom(\bigoplus_{\bfk\in \NN^r}k[x,t]f_\bfk,\,k[\kS])}.
   \]
\end{definition}

Let $k[\ktT]=k[u_0,\dots,u_g]/\cI_{\ktT}$, where $\cI_{\ktT}=(p_1,\dots,p_k)$, for some homogenous polynomials $p_i$. We will write for simplicity $\bfT$ for the list of variables $T_1,T_2,\dots,T_g$.
Every degree $d$ homogenous polynomial $p\in k[u_0,\dots,u_g]=k[u_0,\bfT]$, with $T_i=u_0-u_i$, can be uniquely written as 
\[
  \textstyle
  p=\sum_{n=1}^dp^{(n)}(\bfT)u_0^{d-n},
\]
where $p^{(n)}(\bfT)$ is homogenous of degree $n$. In Subsection \ref{sectZX} we saw that the equations of $\baseMkick$ are given by the ideal
\[
  \cJ:=\big(p_1^{(n)}(\bfT),\dots,p_k^{(n)}(\bfT)\kst n\in \NN\big).
\]
\begin{definition}\label{def i part}
We call $p^{(n)}(\bfT)$ the \emph{degree $n$ part of $p$}.
Let us consider the ideal
\[
  \widetilde{\cJ}:=\cJ\cdot (T_1,T_2,\dots,T_g)+\cJ_1k[\bfT]\subset k[\bfT],
\]
where $\cJ_1:=\big(p_1^{(1)}(\bfT),\dots,p_k^{(1)}(\bfT)\big)$ denotes the ideal generated by the degree one elements. 
Let $W:=\cJ/\widetilde{\cJ}$ be a $\ZZ$-graded vector space $W=\bigoplus_{n\geq 2}W_n$, where  $W_n$ contains the degree~$n$ parts of the polynomials $p\in \cI_{\ktT}$. 
\end{definition}

We have the exact sequence
\begin{equation}\label{eq lw}
  \begin{tikzcd}[column sep = 2em]
0\ar[r]& W\ar[r]& k[\bfT]/\widetilde{\cJ}\ar[r]& k[\bfT]/\cJ\ar[r]& 0.    
  \end{tikzcd}
\end{equation} 
Identifying $t$ with $u_0$, the tensor product of \eqref{eq lw} with $k[\bfx,u_0]$ yields 
\begin{equation}\label{eq exact}
  \begin{tikzcd}[column sep =2em]
0\ar[r]& W\otimes _kk[t,\bfx]\ar[r]& k[t,\bfT,\bfx]/\widetilde{\cJ}\cdot k[t,\bfT,\bfx]\ar[r]& k[t,\bfT,\bfx]/\cJ\cdot k[t,\bfT,\bfx]\ar[r]& 0.    
  \end{tikzcd}
\end{equation}
Using the notation from Subsection \ref{r:projectionOfSum}, let $s=\sum_{\bfk} s_{\bfk}e_{\bfk}\in \cR$, which means $s_{\bfk}\in k[\bfx,t]$ as well as $\psi(s)=0\in k[\bfx,t]$.
In Section \ref{sec flat} we showed that we can lift $s_\bfk$  to $k[u_0,\dots,u_g,\bfx]$, from which we obtain $\widetilde{s}\in \widetilde{\cR}$ such that 
\[
  \textstyle
  o(s):=\sum \widetilde{s}_{\bfk}E_{\bfk}\mapsto 0~~\text{in}~~k[t,\bfT,\bfx]/\cJ\cdot k[t,\bfT,\bfx].
\]
In particular, each relation $s\in \cR$ induces some element $o(s)\in W\otimes_k k[\bfx,t]$, which is well defined after the additional projection to $W\otimes_k k[S]$.
This procedure describes a certain element
\[
  o \in T^2_X\otimes_k W=\h(W^*,T^2_X)
\]
called the \emph{obstruction map} (note that this notation was used in 
\cite{jong} and \cite{alt}  while in \cite{de jong} the obstruction map was defined to be the dual of this). 
From Section \ref{sec flat} (see equation \eqref{eq rel rac}) we obtain that 
\begin{equation}\label{eq obs map gen}
  \textstyle
o(R_{\bfa,\bfk})=\sum_{n\geq 1} x^{\bound(\bfa+\bfk)}t^{n}\otimes h_{n,\bfa,\bfk}(\bfT),
\end{equation}
where $h_{n,\bfa,\bfk}(\bfT)$ is the degree $n$ part of the polynomial 
\begin{equation}\label{eq obs pol}
u^{\tla(\bfa+\bo(\bfk))+\tla(\bfk)}-u^{\tla(\bfa+\bfk)}.
\end{equation}
Note that in \eqref{eq obs map gen} we identified $u_0$ with $t$ and we also use that $x^{\bo(\bfa+\bo(\bfk))}=x^{\bo(\bfa+\bfk)}$
in $k[\ktS]$ by Lemma~\ref{r:projectionOfSum}. 
For our $R=[\un{0},1]\in \kS$ and $n\in \NN$ let us denote by $T^2(-nR)$ the degree $-nR$ part of $T^2_X$.

To prove Theorem \ref{th main} it is enough to prove that the dual of the obstruction map, denoted by $o^*$,  is surjective (see e.g.\ \cite[Section 10]{de jong} for the proof of this statement and note that $o^*\in \Hom((T^2_X)^*,W)$).  To do that we need to understand the equations of $k[\ktT]$. There are two (obvious) types of equations of $k[\ktT]$: the first one we call the loop equations and they are introduced in Section \ref{sec loop}; the second type are the so called local equations introduced in Section \ref{sec loc eq}. In Section \ref{sec-clusterCoord} we prove that the loop and local  equations are in fact all the equations of $k[\ktT]$ by introducing new generators of $\ktT$. 
In order to prove that the dual of the obstruction map is surjective we thus need to prove that the loop and local equations (as elements in $W$) are obtained in the image of $o^*$. For the loop equations this is done in Section \ref{sec loop} (see Corollary \ref{cor loop cor}) and for the local equations this is done in Section \ref{sec loc eq} (see Proposition \ref{pro image h g}).

\section{The loop equations}\label{sec loop}

In this section we generalize the results from \cite[Section 7]{alt} to our setting. Since this section is long and more technical we give some guidance and motivation at the beginning. 
In \cite{alt} and \cite{budapest} the proof of the versality relies on knowing the explicit equations of the versal base space.
We do not have explicit equations but in fact we do not need them, we only need the bi-linearity property, cf.\ Lemma \ref{lem imp lem}.

In Subsection \ref{subsec t2 rev} we analyze the $T^2_X$-module in more details and we introduce the submodule of $(T^2_X)^*$, cf.\ \eqref{eq t2 submod}, which is mapped to our loop equations by $o^*$. We first describe the restriction of the map $o^*$ to this submodule in Proposition \ref{prop des obs map} and prove that all the loop equations are in the image of $o^*$ in Proposition \ref{prop image psi} and Corollary \ref{cor loop cor}.

\subsection{Bi-linearity of the equations}
For each $c\in M$ and a closed path $\un{\mu}$ we define $S^+_{c,\un{\mu}}$ (resp. $S^-_{c,\un{\mu}}$) to be the set of edges 
$d^{ij}$ of $P$, such that $\lan \mu_{ij}d^{ij},c\ran> 0$ (resp. $\lan \mu_{ij}d^{ij},c\ran\leq 0$).
We see by \eqref{2 face equation} and Remark \ref{rem minus} that  
\begin{equation}\label{eq rel 333}
\sum_{d^{ij}\in S^+_{c,\un{\mu}}}\kLp(\mu_{ij}c,d^{ij})-\sum_{d^{ij}\in S^-_{c,\un{\mu}}}\kLp(\mu_{ij}c,d^{ij})=0.
\end{equation}
We call the equations corresponding to \eqref{eq rel 333} \emph{the loop equations} and denote them by
\begin{equation}\label{eq pcmu}
p(\un{\mu},c):=\prod_{d^{ij}\in S^+_{c,\un{\mu}}}u(\mu_{ij}c,d^{ij})-\prod_{d^{ij}\in S^-_{c,\un{\mu}}}u(\mu_{ij}c,d^{ij}).
\end{equation}
\begin{remark}\label{rem brezveze}
  We view \smash{$p(\un{\mu},c)$} as a polynomial in the variables $u_0,\dots,u_g$, which correspond to the generators of $\ktT$.
  There are many different ways to present \smash{$p(\un{\mu},c)$} as a polynomial in $u_0,\dots,u_g$.
  Whenever we say that a property holds for $p\in \cJ$ we mean that it holds for all possible presentations.
  In particular, the bi-linearity of $p$ shown  in Lemma \ref{lem imp lem} holds for all presentations of $p$ as a polynomial in~$\cJ$.
\end{remark}
Let $p^{(n)}(\un{\mu},c)$ denote the degree $n$ part of the polynomial $p(\un{\mu},c)$.

\begin{lemma}\label{lem imp lem}

For $m_1,m_2\in \tail(P)^\vee \cap M$ and a closed path $\un{\mu}$ we have
\begin{equation}\label{eq lin d}
p^{(n)}(\un{\mu},m_1+m_2)=p^{(n)}(\un{\mu},m_1)+p^{(n)}(\un{\mu},m_2)\in W_n.
\end{equation}
For two closed paths $\un{\mu}^1,\un{\mu}^2$ and $m\in \tail(P)^\vee \cap M$ it holds that 
\begin{equation}\label{eq lin p}
p^{(n)}(\un{\mu}^1+\un{\mu}^2,m)=p^{(n)}(\un{\mu}^1,m)+p^{(n)}(\un{\mu}^2,m)\in W_n.
\end{equation}
\end{lemma}
\begin{proof}
  We define four sets of edges based on the sign of the pairing with $m_1+m_2$, $m_1$ and $m_2$:
  \begin{eqnarray*}
E_1:=&&\hspace{-1.4em}\{d^{ij}\in \cedges(P) \kst \braket{ \mu_{ij}d^{ij},m_1+m_2}<0,~\braket{ \mu_{ij}d^{ij},m_1}>0,~\braket{ \mu_{ij}d^{ij},m_2}<0\},\\
E_2:=&&\hspace{-1.4em}\{d^{ij}\in \cedges(P) \kst \braket{ \mu_{ij}d^{ij},m_1+m_2}>0,~\braket{ \mu_{ij}d^{ij},m_1}<0,~\braket{ \mu_{ij}d^{ij},m_2}>0\},\\
E_3:=&&\hspace{-1.4em}\{d^{ij}\in \cedges(P) \kst \braket{ \mu_{ij}d^{ij},m_1+m_2}<0,~\braket{ \mu_{ij}d^{ij},m_1}<0,~\braket{ \mu_{ij}d^{ij},m_2}>0\},\\
E_4:=&&\hspace{-1.4em}\{d^{ij}\in \cedges(P) \kst \braket{ \mu_{ij}d^{ij},m_1+m_2}>0,~\braket{ \mu_{ij}d^{ij},m_1}>0,~\braket{ \mu_{ij}d^{ij},m_2}<0\}.
  \end{eqnarray*}
Straightforward computation shows that for each non-lattice vertex $v\in P$ there exist $n_v,m_v\in \NN$ such that in $\cI_{\ktT}$ the following holds:
\begin{equation}\label{eq pc1c2}
p(\un{\mu},m_1+m_2)\prod_{d^{ij}\in E_1\cup E_2}u(\mu_{ij}m_1,d^{ij})\prod_{d^{ij}\in E_3\cup E_4}u(\mu_{ij}m_2,d^{ij})\prod_{v}u(s_v)^{m_v}=
\end{equation}
\[
  =\frac{1}{2}p(\un{\mu},m_1)\Big(\prod_{d^{ij}\in S^+_{m_2,\un{\mu}}}u(\mu_{ij}m_2,d^{ij})+\prod_{d^{ij}\in S^-_{m_2,\un{\mu}}}u(\mu_{ij}m_2,d^{ij})\Big)\prod_{v}u(s_v)^{n_v}+
\]
\[
  \frac{1}{2}p(\un{\mu},m_2)\Big(\prod_{d^{ij}\in S^+_{m_1,\un{\mu}}}u(\mu_{ij}m_1,d^{ij})+\prod_{d^{ij}\in S^-_{m_1,\un{\mu}}}u(\mu_{ij}m_1,d^{ij})\Big)\prod_{v}u(s_v)^{n_v},
\]
from which \eqref{eq lin d} follows after looking at the degree $n$ part of the above equation taken modulo $\widetilde{\cJ}$. We can prove \eqref{eq lin p} in a similar way, so we omit the proof. 
\end{proof}

\subsection{The module $T^2_X$ revisited}\label{subsec t2 rev}
We recall the following from \cite[Section 5.5]{klaus}.
Let $\sigma=\cone(P)$ be generated by $a^i\in N$ and let $E$ denote the Hilbert basis of $S=\sigma^\vee\cap (M\oplus \ZZ)$. We consider the canonical surjection $p:\ZZ^E\to M\oplus \ZZ$. Its kernel is a $\ZZ$-module $L(E):=\ker p$ which encodes the relations among elements in $E$. 
\begin{definition}\label{def lai}
For $R\in M\oplus \ZZ$ consider 
\[
  E_{a_i}^R:=E_i^R:=\{e\in E \kst \lan a^i,e \ran<\lan a^i,R \ran\}.
\]
For a subface $\tau\leq \sigma$ we define $E^R_\tau:=\bigcap_{a^i\in \tau}E^R_{i}$ and $L(E^R_\tau):=L(E)\cap \ZZ^{E^R_{\tau}}$.  Moreover, for $p\in \NN$ we define
\[
  L(E^R)_p:=\bigoplus_{\tau\leq \sigma,\dim\tau=p}L(E^R_\tau).
\]
 After defining $L(E^R)_0:=\bigcup_{i}E^R_{a^i}$ we get a complex $L(E^R)_\kbb$
 with the usual differentials. Let us define $L_k(E^R_\kbb):=L(E^R)_\kbb\otimes_\ZZ k$.
\end{definition}
We have an exact sequence
\begin{equation}\label{eq ex sequence}
0\to L_k(E^R_\kbb)\to k^{E^R_\kbb}\to \span_k E^R_\kbb\to 0.
\end{equation}
Let us consider the first homology group of the complex $L_k(E^R)_\kbb$:
\[
  H_1(L_k(E^R_\kbb))=\Bigg(\frac{\ker\Big( \bigoplus_{i}L_k(E^R_i)\to L_k(E)\Big)}{\im\Big(\bigoplus_{\lan a^i,a^k\ran\leq \sigma}L_k(E_i^R\cap E_k^R)\to \bigoplus_{i}L_k(E^R_i)\Big)} \Bigg),
\]
which is isomorphic to $H_2(\span_k E^R_\kbb)$ since $H_i(k^{E^R_\kbb})=0$ for $i\geq 1$.
In \cite[Section 5.5]{klaus} was proven that 
\begin{equation}\label{eq t2 submod}
  H_1(L_k(E^R_\kbb))\otimes_\ZZ k\subset \Big(T^2(-R)\Big)^*.
\end{equation}

Recall the elements $c_1,\dots,c_r$ appearing in the Hilbert basis of $S$ and the paths $\un{\lam}(a)$, $\un{\lam}^c(a)$ and $\lam^c(a)$ for $a,c\in \tail(P)^\vee$ (cf. Definition~\ref{def path}). For each vertex $v^i$ of $P$ we get the corresponding generator $a^i$ of $\sigma$. For a vertex $v$ of $P$ and $c\in \tail(P)^\vee$ we define similar paths 
\begin{eqnarray*}
  \un{\lam}(v)&:=& [\text{some path } v_*\leadsto v] = [\lam_1(v),\dots,\lam_r(v)]\in\Z^\kPr,\\
  \un{\mu}^\kc(v)&:=&[\text{some path } v\leadsto v(\kc)\text{ such that }\mu^c_{i}(v)\braket{ d^{i},\kc}\leq 0~\forall~d^{i}] = [\mu^c_1(v),\dots,\mu^c_r(v)]\in\Z^\kPr,\text{~and}\\
\un{\lam}^c(v)&:=&\un{\lam}(v)+\un{\mu}^c(v).
\end{eqnarray*}
In particular, $\un{\lam}(a)=\un{\lam}(v(a))$, $\un{\mu}^c(a)=\un{\mu}^c(v(a))$ and $\un{\lam}^c(a)=\un{\lam}^c(v(a))$.
For $n\in \NN$, $n\geq 2$, we define the map:
\[
  \begin{tikzcd}[row sep = 0ex]
    \psi_i^{(n)}:L_k(E^{nR}_{a^i})\ar[r]
    &
    \makebox[\widthof{$\sum_{j=1}^rq_jp^{(n)}(\un{\lam}^{c_j}(v^i)-\un{\lam}(v(c_j)),c_j).$}][l]{$W_n$}
    \\
    \makebox[\widthof{$\psi_i^{(n)}:L_k(E^{nR}_{a^i})$}][r]{$\un{q}$}\ar[mapsto,r]
    &
    \textstyle \sum_{j=1}^rq_jp^{(n)}(\un{\lam}^{c_j}(v^i)-\un{\lam}(v(c_j)),c_j).
  \end{tikzcd}
\]

\begin{lemma}\label{lem defin lem}
  The maps $\psi_i^{(n)}$ induce the linear map
  $
  \begin{tikzcd}
    \psi^{(n)}: H_1(L_k(E^{nR}_\kbb))\ar[r]& W_n. 
  \end{tikzcd}
  $
\end{lemma}
\begin{proof}
  We need to show that for every face $\span_{\RR_{\geqslant 0}}\{a^i,a^j\}<\sigma$ the maps  $\psi_i^{(n)}$ and $\psi_j^{(n)}$ agree on $L(E^{nR}_{a^i}\cap E^{nR}_{a^j})$. Let us write $\un{\lam}^k(v):=\un{\lam}^{c_k}(v)$ and compute using Lemma \ref{lem imp lem} that  
  \[
    \textstyle
  \psi_i^{(n)}(\un{q})-\psi_j^{(n)}(\un{q})=\sum_{k=1}^rq_kp^{(n)}(\un{\lam}^{k}(v^i)-\un{\lam}^{k}(v^j),c_k).
\]
Denoting by $\rho^{ij}$ the path consisting of the single edge running from $v^i$ to $v^j$ we see by Lemma \ref{lem imp lem} that for $\un{q}\in L(E^{nR}_{a^i}\cap E^{nR}_{a^j})$ we have
\[
  \psi_i^{(n)}(\un{q})-\psi_j^{(n)}(\un{q})=\sum_{k=1}^rq_kp^{(n)}(\un{\lam}(v^i)-\un{\lam}(v^j)+\rho^{ij},c_k)+\sum_{k=1}^r q_kp^{(n)}(\un{\mu}^k(v^i)-\un{\mu}^k(v^j)-\rho^{ij},c_k)=0.
\]
Indeed, the first sum is zero since $\sum_{k=1}^rq_kc_k=0$ and the second sum is zero since $\un{q}\in L(E^{nR}_{a^i}\cap E^{nR}_{a^j})$, from which we can easily compute that the degree of $p(\un{\mu}^k(v^i)-\un{\mu}^k(v^j)-\rho^{ij},c_k)$ is strictly smaller than $n$, which concludes the proof.
\end{proof}

\subsection{The restriction of the obstruction map}

\begin{proposition}\label{prop des obs map}
The map  $\sum_{n\geq 1}\psi^{(n)}$ is equal to the dual of the obstruction map $o^*$ restricted to
  \[\textstyle
    \bigoplus_{n\geq 1}H_1(L_k(E^{nR}_\kbb))\subset (T^2_X)^*.
  \]
\end{proposition}
\begin{proof}
  Let $\bfa=(k^\bfa_1,\dots,k^\bfa_r)\in \NN^r$ and $\bfk=(k^{\bfk}_1,\dots,k^\bfk_r)\in \NN^r$. Let
  \begin{eqnarray*}
    c_{\bfa}&:=&\textstyle\sum_{j=1}^rk^\bfa_jc_j,\\
    c_{\bfk}&:=&\textstyle\sum_{j=1}^rk^\bfk_jc_j, 
  \end{eqnarray*}  
  where $c_j$ appear in the Hilbert basis of $S$, see \eqref{eg hilbbas}.
We denote $\bo(\bfa+\bfk)=(k^{\bo(\bfa+\bfk)}_1,\dots,k^{\bo(\bfa+\bfk)}_r)$ and $\bo(\bfk)=(k_1^{\bo(\bfk)},\dots,k_r^{\bo(\bfk)})$.
Recall the linear relation $R_{\bfa,\bfk}$, which can also be rewritten as 
\begin{eqnarray}
  \label{eq eq druga}
  R_{\bfa,\bfk}&=&\phantom{-}x^{\bfa+\bfk}-x^{\bfa+\bo(\bfk)}t^{\la(\bfk)}-x^\bfa\big(x^\bfk-x^{\bo(\bfk)}t^{\la(c)}\big)+t^{\la(\bfk)}x^{\bo(\bfa+\bfk)}-\\
  \nonumber &&-x^{\bo(\bo(\bfk)+\bfa)}t^{\la(\bo(\bfk)+\bfa)+\la(\bfk)}-t^{\la(\bfk)}\big(x^{\bo(\bfk)+\bfa}-x^{\bo(\bo(\bfk)+\bfa)}t^{\la(\bo(\bfk))+\bfa} \big).
\end{eqnarray}
Let us denote $s_{\bfa}:=\sum_{j=1}^rk^{\bfa}_js_j$ and $s_{\bfk}:=\sum_{j=1}^rk^{\bfk}_js_j$, where $s_j$ are the Hilbert basis elements, see \eqref{eg hilbbas}. 
Using \cite[Theorem 3.5]{cup} we can find an element of
\[
  \h(\cR/\cR_0,W_n\otimes \cO(X))
\]
representing $\psi^{(n)}.$ Using our notation we can easily verify that it sends relation $R_{\bfa,\bfk}$ to 
\begin{equation}\label{eq prvi dokaz}
\left\{
\begin{array}{ll}
\Big(\psi^{(n)}_{v(c_\bfk)}(\bfk-\bo(\bfk))-\psi^{(n)}_{v(c_\bfa+c_\bfk)}(\bfk-\bo(\bfk))\Big)x^{\bfa+\bfk-nR}& \text{ if }\la(s_\bfa+s_\bfk)\geq n, \\
\rule{0pt}{2em}0&\text{ otherwise}.
\end{array}
\right.
\end{equation}
Moreover, we define 
\begin{eqnarray*}
  t'_{c_\bfk,j}&:=&\textstyle\lan v_{*},c_j\ran s_{v_{*}}+\sum_\nu\delta_{j,\nu}(c_\bfk)\kLp\big(\lam^j_\nu(c_{\bfk})c_j,d^\nu\big)\in \cT_\ZZ^*(P),\\
  t'_{c_{\bfa}+c_{\bfk},j}&:=&\textstyle\lan v_{*},c_j\ran s_{v_{*}}+\sum_\nu\delta_{j,\nu}(c_\bfa+c_\bfk)\kLp\big(\lam^j_\nu(c_{\bfa}+c_{\bfk})c_j,d^\nu\big)\in \cT_\ZZ^*(P),
\end{eqnarray*}
where 
\[
  \delta_{j,\nu}(c_\bfk):=\left\{\begin{array}{ll}
~~~1& \text{ if }\lan c_j,\lam^j_\nu(c_\bfk) d^\nu\ran>0\\
-1&\text{ if }\lan c_j,\lam^j_\nu(c_\bfk) d^\nu\ran \leq0,~~~~~~~
 \end{array}
 \right.
 \delta_{j,\nu}(c_\bfa+c_\bfk):=\left\{\begin{array}{ll}
~~~1& \text{ if }\lan c_j,\lam^j_\nu(c_\bfa+c_\bfk) d^\nu\ran>0\\
-1&\text{ if }\lan c_j,\lam^j_\nu(c_\bfa+c_\bfk) d^\nu\ran \leq0.
 \end{array}
 \right.
 \]
Using Proposition \ref{prop tildelam} we see that the polynomial  
\[
  u^{\tla(\bfa+\bo(\bfk))+\tla(\bfk)}-u^{\tla(\bfa+\bfk)},
\]
appearing in
\eqref{eq obs pol}, equals
$u^{e_1}-u^{e_2},$
where
\begin{eqnarray*}
  e_1&=&\textstyle\sum_{j=1}^r\big((k_j^{\bfa}+k_j^{\bo(\bfk)}-k_j^{\bo(\bfa+\bo(\bfk))})t'_{c_\bfa+c_\bfk,j}+(k_j^\bfk-k_j^{\bo(\bfk)})t'_{c_\bfk,j},\\
  e_2&=&\textstyle\sum_{j=1}^r(k_j^{\bfa}+k_j^{\bfk}-k_j^{\bo(\bfa+\bfk)})t'_{c_{\bfa}+c_{\bfk},j}.
\end{eqnarray*}
By the equation \eqref{eq boundary} we see that $c_{\bfk}=\sum_{j=1}^rk_j^\bfk=\sum_{j=1}^rk_j^{\bo(\bfk)}$ and by Lemma \ref{r:projectionOfSum} it holds that
\[
  \bo(\bfa+\bo(\bfk))=\bo(\bfa+\bfk).
\]
Using also Lemma \ref{lem imp lem} we see that in $W_n$ it holds that
\[
  \textstyle
  \psi^{(n)}_{v(c)}(\bfk-\bo(\bfk))-\psi^{(n)}_{v(a+c)}(\bfk-\bo(\bfk))=
\sum_j(k^\bfk_j-k^{\bo(\bfk)}_j)p^{(n)}\big(\un{\lam}^{j}(c_\bfk)-\un{\lam}^j(c_\bfa+c_\bfk),c_j\big).
\]
Since $e_1-e_2=(k_j^{\bfc}-k_j^{\bo(\bfc)})(t'_{c_\bfa+c_\bfk,j}-t'_{c_\bfk,j})$, the proof now follows.
\end{proof}

\begin{proposition}\label{prop image psi}
The image of the map $\psi^{(n)}$ contains the equations $p^{(n)}(\eps,c)$ for all bounded $2$-faces $\eps$ and $c\in \tail(\kP)^\vee\cap M$.
\end{proposition}
\begin{proof}
Let us recall the isomorphism of homology groups $H_1(L_k(E^{nR})_\kbb)\cong H_2(\span_k E^{nR}_\kbb)$ explained after the exact sequence \eqref{eq ex sequence} ($n\in \NN$). Let the rank of the lattice $M$ be $d-1$ and thus $\span_k(M\oplus \ZZ)\cong k^d$. For $n\geq 2$ we have 
\[
  H_2(\span_k E^{nR}_\kbb)=\frac{\ker[\bigoplus_{\lan a^i,a^j\ran<P}k^d\to \bigoplus_{a^i<P}k^{d}]}{\im[\bigoplus_{\eps<P} \span_k(\cap_{a^i\in \eps}E_i^{nR})\to \bigoplus_{\lan a^i,a^j\ran<P}k^{d}]}.
\]
Indeed, $\span_k E^{nR}_{a^i}\cong k^d$ clearly holds for all rays $a^i\in \sigma$ and
\begin{equation}\label{eq lep eq}
\span_{k}(E^{nR}_{a^i}\cap E^{nR}_{a^j})\cong k^d
\end{equation}
holds for all 2-faces $\span_{\RR_{\geqslant 0}}\{a^i,a^j\}$ of $\sigma$ since the lattice length of all edges is smaller than $2$ because the semigroups $\ktT_d$ are generated by degree $1$ elements, see Proposition \ref{prop-degtt-kshort} and Remark \ref{rem klaus remark}.

Clearly it holds that 
\[
  \ker[\bigoplus_{\lan a^i,a^j\ran<P}k^d\too \bigoplus_{a^i<P}k^{d}]\cong \im[\bigoplus_{\eps<P,\dim \eps=2}k^d\too \bigoplus_{\lan a^i,a^j\ran<P}k^d]
\]
since the complex $\bigoplus_{\tau<P,\dim \tau=\kbb}k^d$ is acyclic in degrees $\geq 1$. Thus we have a surjection 
\[
  \begin{tikzcd}
  g:\bigoplus_{\eps<P,\dim\eps=2}k^d\ar[twoheadrightarrow,r]& H_2(\span_k E^{nR}_\kbb)\cong H_1(L_k(E^{nR})_\kbb).    
  \end{tikzcd}
\]

In the following we will explicitly describe the map $g$. 
After choosing a $2$-face with oriented edges $d^1,\dots,d^m$ we represent $[c,\ketaZ(c)]$ as a linear combination of elements of $E_{a^i}^{nR}\cap E_{a^{i+1}}^{nR}$:
\[
  \textstyle
  [c,\ketaZ(c)]=\sum_jq_{i,j}[c_j,\ketaZ(c_j)]+q_i(\un{0},1),
\]
and $q_{i,j}\ne 0$ implies $[c_j,\ketaZ(c_j)]\in E^{nR}_{a^i}\cap E^{nR}_{a^{i+1}}$.
This corresponds to the lifting of an element from $\span_k E^{nR}_2$ to $k^{E^{nR}_2}$.
From this we get an element in
\[
  \textstyle
  \ker\big(\bigoplus_iL(E_{a^i}^{nR})\too L(E)\big)
\]
whose $i$-th summand is the linear relation 
\[
  \textstyle
  \sum_j(q_{i,j}-q_{i-1,j})[c_j,\ketaZ(c_j)]+(q_i-q_{i-1})(\un{0},1)=0.
\]
Thus we explicitly describe the map $g$.

Now we will check that $(\psi^{(n)}\circ g)[c,\ketaZ(c)]=p^{(n)}(\eps,c)$ holds:
 \begin{eqnarray*}
 (\psi^{(n)}\circ g)[c,\ketaZ(c)]&=&\textstyle\sum_{i=1}^m\sum_{j=1}^r(q_{i,j}-q_{i-1,j})p^{(n)}(\un{\lam}^{c_j}(v^i)-\un{\lam}(v(c_j)),c_j)\\
\rule{0pt}{1.5em} &=&\textstyle \sum_{i=1}^m\sum_{j=1}^r p^{(n)}(\un{\lam}^{c_j}(v^i)-\un{\lam}^{c_j}(v^{i+1}),q_{i,j}c_j),
 \end{eqnarray*}

 where in the last equality we used bi-linearity of $p^{(n)}$ proven in Lemma \ref{lem imp lem}. Now as in Lemma \ref{lem defin lem} we introduce the path $\rho^{i}$ consisting of the single edge running from $a^i$ to $a^{i+1}$ and compute that
   \begin{eqnarray*}
  (\psi^{(n)}\circ g)[c,\ketaZ(c)]
     &=&\textstyle\sum_{i=1}^m\sum_{j=1}^r p^{(n)}(\un{\lam}(v^i)+\un{\mu}^j(v^i)-\un{\lam}(v^{i+1})-\un{\mu}^j(v^{i+1}),q_{i,j}c_j)\\
      &=&\rule{0pt}{1.5em}\textstyle\sum_{i=1}^m p^{(n)}(\un{\lam}(v^i)-\un{\lam}(v^{i+1})+\rho^{i},\sum_{j=1}^rq_{i,j}c_j)+\\
      &&\rule{0pt}{1.5em}\textstyle+\sum_{i=1}^m\sum_{j=1}^rp^{(n)}(\un{\mu}^j(v^i)-\un{\mu}^j(v^{i+1})-\rho^{i},q_{i,j}c_j)\\
      &=&\rule{0pt}{1.5em}\textstyle\sum_{i=1}^m p^{(n)}(\un{\lam}(v^i)-\un{\lam}(v^{i+1})+\rho^i,\sum_{j=1}^rq_{i,j}c_j).
   \end{eqnarray*}
We used in the computation above that $p^{(n)}(\un{\mu}^j(v^i)-\un{\lam}^j(v^{i+1})-\rho^i,q_{i,j}c_j)=0$, which can be verified the same way as in the proof of Lemma \ref{lem imp lem}. From this it follows that 
\[
  \textstyle
 (\psi^{(n)}\circ g)[c,\ketaZ(c)]=\sum_{i=1}^mp^{(n)}(\un{\lam}(v^i)-\un{\lam}(v^{i+1})+\rho^i,c)=\sum_{i=1}^mp^{(n)}(\rho^i,c)=p^{(n)}(\eps,c).
 \]
\end{proof}

\begin{corollary}\label{cor loop cor}
The image of $o^*$ contains all the  degree $n$ parts of all loop equations \eqref{eq pcmu}.
\end{corollary}
\begin{proof}
Let $p(\un{\mu},c)$ be a loop equation. By Lemma \ref{lem imp lem} we can write its degree $n$ part as $\sum_{j=1}^kp^{(n)}(\eps_j,c)$, where $\eps_j$ are bounded $2$-faces. We conclude by Propositions \ref{prop des obs map} and \ref{prop image psi}.
\end{proof}

\begin{remark}\label{rem finrem}
Note that in the proof of Proposition \ref{prop image psi} we really need the assumption that the semigroups $\ktT_d$ are generated by degree $1$ elements. If we would only assume that $\ktT$ is generated by degree $1$ elements, then the crucial equation \eqref{eq lep eq} in the proof above might not be satisfied. See Example \ref{ex house} and consider the edge $d_1$ that defines the 2-face $\span_{\RR_{\geqslant 0}}\{ a^1,a^2\}$ of $\sigma$. Here we have that $\lan a^i,R \ran=1$ and thus 
\[
  \span_{k}(E^{2R}_{a^1}\cap E^{2R}_{a^2})\cong \span_k((a^1)^\perp\cap (a^2)^\perp,R)\cong k^2.
\]
The rank of $M\oplus \ZZ$ is $3$ here so we see that the equation \eqref{eq lep eq} is not satisfied in this example.
\end{remark}

\section{The local equations}\label{sec loc eq}

Let us now explicitly write the set of generators from Proposition \ref{pro 2dim fin gen}. For each compact edge $d^{ij}=v^j-v^i$
let
\[
  k_{ij}:=\min\{ |\lan c,d^{ij}\ran| \,;\, c\in M, \lan c,d^{ij}\ran\ne 0\}.
\]
We choose $c_{ij}\in M$ 
such that $\lan c_{ij},d^{ij}\ran=k_{ij}$.  We define $m_{ij}$ to be the minimal natural number such that $m_{ij}\lan c_{ij},v^i\ran,m_{ij}\lan c_{ij},v^{j}\ran\in \ZZ$. 

From Proposition \ref{pro 2dim fin gen} we see that the following set 
is a generating set for $\ktT$:
\begin{equation}\label{eq gen ktt}
\big\{\kLp(k c_{ij},d^{ij}),s_v\kst d^{ij}\in \cedges(P),k\in \{\pm 1,\dots,\pm m_{ij}\},v\in \Vrtx(P)\big\}.
\end{equation}
We fix an edge $d^{ij}$ and let $n$ be the maximal natural number, such that the degree of $\kLp(n c_{ij},d^{ij})$ equals~$1$.
Similarly, let $m$ be the maximal natural number, such that the degree of $\kLp(-m c_{ij},d^{ij})$ is $1$.
Recall the definition of the sub-monoid $\ktT_{ij}:=\ktT_{d^{ij}}$ from Definition \ref{def submon}.
Since $\ktT$ is generated by degree $1$ elements, then the following elements are the minimal generators of $\ktT_{ij}$:
\[
  s_i,s_j,~~~\kLp(c_{ij},d^{ij}),\dots,\kLp(nc_{ij},d^{ij}),~~~\kLp(-c_{ij},d^{ij}),\dots,\kLp(-mc_{ij},d^{ij}).
\]

Let the polytope $Q_{ij}$ be the convex hull of the vertices $(0,0), (0,1), (n,0), (m,1)$. Let $C(Q_{ij})$ be the cone over this polytope, i.e., generated by $(0,0,1), (0,1,1), (n,0,1), (m,1,1)$. Denote the following monoid  by $T'_{ij}:=\ZZ^3\cap C(Q_{ij})$.
From the description of these elements in Section \ref{sec sem ktt}, one can see that the monoid  $\ktT_{ij}$ is isomorphic to $T'_{ij}$: the isomorphism is given by 
\[
  s_i\mapsto (0,0,1),~~s_{j}\mapsto (0,1,1),~~\kLp(ac_{ij},d^{ij})\mapsto (a,0,1),~~\kLp(-bc_{ij},d^{ij})\mapsto (b,1,1),
\]
for $a=1,\dots,n$ and $b=1,\dots,m$.

\begin{corollary}\label{cor eq of tij}
The affine variety $\spec k[\ktT_{ij}]$ is given by the equations 
\begin{eqnarray}
  \label{eq first e}
  x_{k_1}x_{k_2}-x_{l_1}x_{l_2}&=&0, ~~~1\leq k_1+k_2=l_1+l_2\leq n,\\
  \label{eq second e}
  y_{k_1}y_{k_2}-y_{l_1}y_{l_2}&=&0, ~~~1\leq k_1+k_2=l_1+l_2\leq m,\\
  \label{eq third e}
  x_{k}y_0-y_kx_0&=&0, ~~~1\leq k\leq \min\{n,m\},
\end{eqnarray}
where, for $k\geq 0$,
$x_k$ and $y_k$  correspond to the generators $\kLp(kc_{ij},d^{ij})$ and $\kLp(-kc_{ij},d^{ij})$, respectively.
In particular, in this notation, $x_0$ corresponds to $s_i$ and $y_0$ corresponds to $s_j$.
\end{corollary}

\begin{proposition}\label{pro image h g}
The image of $o^*$ contains all the  degree $n$ parts of all the equations in Corollary \ref{cor eq of tij}.
\end{proposition}
\begin{proof}
Let us first consider  the case when $\kP$ is a bounded $1$-dim polytope, i.e. a line segment with $0\in \inter(\kP)$. More precisely, we assume that
 $\kP=d=[v,w]=[-\frac{a_1}{b_1},\frac{a_2}{b_2}]$ with $a_i,b_i>0$ and such that $\frac{a_1}{b_1}$ and $\frac{a_2}{b_2}$ are not integers (the same setting as in Example \ref{ex dij}). The general case follows easily from this one using the proof of Lemma \ref{lem pom lem kon}. 
 We have $\sigma^\vee=\langle (-b_2,a_2),(b_1,a_1)\rangle$ with the
 Hilbert basis equal to
 \[
   \{[-b_2,\ketaZ(-b_2)],\dots,[-1,\ketaZ(-1)],[1,\ketaZ(1)],\dots,[b_1,\ketaZ(b_1)]\}.
 \]
Thus $r$ is equal to $b_1+b_2$ in this case and we write $e_{-b_2},e_{-b_2+1},\dots,e_{-1},e_{1},e_{2}\dots,e_{b_1}$ for the components of $\bfa,\bfk\in \NN^r$.

Using the notation from Corollary \ref{cor eq of tij} we have that
\[
  \ketaZ(1)=\cdots =\ketaZ(n)=1,~~~~\ketaZ(-1)=\cdots =\ketaZ(-m)=1.
\]
Let $k_1, k_2,l_1,l_2\in \NN$ be such that $1\leq k_1+l_1=k_2+l_2\leq n$ and $k_1\leq l_1$, $k_2\leq l_2$. We define 
\[
  \bfk_1:=e_n+e_{-k_1},~~\bfa_1:=e_{-l_1},~~\bfk_2:=e_n+e_{-k_2},~~\bfa_2:=e_{-l_2}.
\] 
For $i=1,2$ we have
\[
  \tla(\bfk_i)=\kteta(n,-k_i)=\kLp(k_i,d),
\]
where the last equality was proven in Example \ref{ex dij}.
Moreover,  for $i=1,2$ we have
\[
  \tla(\bfa_i+\tbo(\bfk_i))=\tla(e_{-l_i}+e_{n-k_i})=\kLp(l_i,d).
\]
Since 
\(
  \tla(\bfa_1+\bfk_1)=\tla(\bfa_2+\bfk_2)
\)
we obtain 
\begin{eqnarray*}
  o^*(R_{\bfa_1,\bfk_1}-R_{\bfa_2,\bfk_2})&=&\Big(u^{\tla(\bfa_1+\tbo(\bfk_1))}u^{\tla(\bfk_1)}-u^{\tla(\bfa_1+\bfk_1)} \Big)-\Big(u^{\tla(\bfa_2+\tbo(\bfk_2))}u^{\tla(\bfk_2)}-u^{\tla(\bfa_2+\bfk_2)} \Big)\\
&=&  u^{\kLp(k_1,d)}u^{\kLp(l_1,d)}-u^{\kLp(k_2,d)}u^{\kLp(l_2,d)},
\end{eqnarray*}
from which we see that the $n$-th parts of the equations \eqref{eq first e} are in the image of $o^*$.
Defining 
\[
  \bfk'_1:=e_{-m}+e_{k_1},~~\bfa'_1:=e_{l_1},~~\bfk'_2:=e_{-m}+e_{k_2},~~\bfa'_2:=e_{l_2}
\]
gives us as above that 
\[
  o^*(R_{\bfa'_1,\bfk'_1}-R_{\bfa'_2,\bfk'_2})=u^{\kLp(-k_1,d)}u^{\kLp(-l_1,d)}-u^{\kLp(-k_2,d)}u^{\kLp(-l_2,d)},
\]
from which we see that the $n$-th parts of the equations \eqref{eq second e} are in the image of $o^*$.
Finally, denoting $\widetilde{m}:=\min\{m,n\}$ and for $k=1,\dots,\widetilde{m}$ defining 
\[
  \bfk''_1:=e_{\widetilde{m}}+e_{-k},~~\bfa''_1:=e_{1},~~\bfk''_2:=e_{-\widetilde{m}}+e_{k},~~\bfa''_2:=e_{-1}
\]
give us as above that the $n$-th parts of the equations \eqref{eq third e} are in the image of $o^*$.
\end{proof}
So to prove our main Theorem \ref{th main} we only need to show that the loop and local equations are generating all the equations of $k[\ktT]$.
This is done in the next section. 
\section{New generators of $\ktT$}\label{sec-clusterCoord}
\newcommand{\kPe}{{P^{(1)}}}
\newcommand{\krho}{\rho}
\newcommand{\krhoA}{\rho^{A}}
\newcommand{\krhoB}{\rho^{B}}
\newcommand{\krhoD}{\rho^{D}}
\newcommand{\krhoN}{\rho^{N}}
\newcommand{\kRp}{\wt{R}}

\subsection{Decomposing the 1-skeleton of $\kP$}\label{subsec-decSkeleton}
We start with some graph theoretic considerations. Denote by $\kPe$
the compact part of the 1-skeleton of the polyhedron $\kP$. It 
splits into the vertices and the interior parts of the compact edges.
We extend this to an abstract graph $\kP'$ by adding abstract edges 
$e(v,w)$ between vertices $v,w\in\kP$ such that
$[v,w]\leq\kP$ is an ordinary edge with $[v,w]\cap\kQuotD=\emptyset$. 
This graph contains the following subsets:
\begin{itemize}
\item[(V)] consisting of the vertices $v\in\kP$ being not contained in 
the lattice $\kQuotD$, the new abstract edges $e(v,w)$, 
and of the \se\ half open edges $[v,w)$ and
\item[(D)] consisting of the remaining open edges $(v,w)$, i.e.\ those such
that neither $[v,w)$ nor $(v,w]$ is \se.
\end{itemize}
While $D$ consists of isolated (open) edges, the set $V$ contains 
connected clusters,
made from vertices and half open edges. This leads to the next step.
We denote by 
\begin{itemize}
\item[(A)]
the set of those connected components of $V$ that do not contain
\se\ half open edges, i.e.\ of those consisting only of 
(non-lattice) vertices $v$ and new abstract edges $e(v,w)$ and 
\item[(B)]
the set of remaining connected components of $V$.
In particular, every component from $B$ contains at least 
one \se\ half open edge.
\end{itemize}
Altogether, our graph $\kP'\supseteq\kPe$ splits into a 
disjoint union of elements of $A$, $B$, $D$, 
and of the lattice vertices of the original polyhedron $\kP$.
The latter set might be denoted by $N$.

\begin{example}
Let $P:=\conv\{v_1,v_2,v_3,v_4\}$ with 
$v_1=(-1,0)$, $v_2=(1,0)$, $v_3=(\frac{2}{3},\frac{1}{2})$, $v_4=(-\frac{1}{6},\frac{1}{2})$ and let $$d_1:=(v_1,v_2),~~~d_2:=(v_2,v_3),~~~d_3:=(v_3,v_4),~~~d_4:=(v_4,v_1)$$ be the open edges between our vertices. 
We see that $[v_3,v_2)$ and $[v_4,v_1)$ are short half open edges and that there are no other short (half open) edges (see Example \eqref{ex-CQS}). Thus $$A=\{e(v_3,v_4)\},~~~B=\{v_3,v_4,d_2,d_4\},~~~D=\{d_1,d_3\}.$$
\end{example}

\subsection{New variables}\label{subsec-newVar}
We are going to replace the old variables $s_v$ (for vertices $v\in\kP$)
and $t_d$ (for edges $d\leq\kP$) by new ones.
They will be denoted by
$\krhoA_a$ (for $a\in A$),
$\krhoB_b$ (for $b\in B$), and
$\krhoD_d$ (for $d\in D$), and they are defined as follows:
\begin{itemize}
\item[($\krhoA$)]
If $a\in A$, then this cluster is not incident with any of
the ordinary edges, but with some of the vertices $v\in\kP$. We denote
$\krhoA_a:=s_v$ (for any such $v$).
\item[($\krhoD$)]
If $d\in D$, then this points to a single edge $d\leq\kP$. 
We set $\krhoD_d:=t_d$.
\item[($\krhoB$)]
This is the only type of the three clusters containing both
vertices $v\in\kP$ and ordinary edges $d\in\kP$. 
If $b\in B$, then we set
$\krhoB_b:=s_v=t_d$ (for any such $v,d$ being contained in $b$).
\end{itemize}
If $\mu\in\Z^{\kpr}$ is induced from a closed path along the
compact edges, e.g.\ from the boundary of a compact 2-face $F\leq\kP$, then,
for any $\kc\in\kQuot$, we used to have the equation
\[
\sum_{\nu=1}^\kpr \mu_\nu\cdot \langle c,d^\nu\rangle\cdot t_{d^\nu}=0.
\]
This turns into the relation
\[
\sum_{d\in D} \mu_d\cdot \langle c,d\rangle\cdot \krhoD_{d}
+
\sum_{b\in B} \big\langle c,\sum_{d\in b} \mu_d \,d\big\rangle 
\cdot \krhoB_{b}=0,
\]
which we will call the loop relation $R(\mu,\kc)$.

That is, compared with \cite{alt}, we keep the equations
for $t_d=\krho_d$ (appearing as $\krhoB_d$ or $\krhoD_d$) along 2-faces $F$. 
However, as in \cite{budapest}, some of the variables are 
forced to become equal (the former $t_d=\krhoB_d$ corresponding to those $d$ 
being contained in some joint $b\in B$ become $\krhoB_b$),
and now, beyond \cite{budapest}, 
we also have additional free variables $\krhoA_a$ (for $a\in A$) not appearing in the loop relations.

\subsection{New generators for $\ktT$}\label{subsec-newGenTt}
By Proposition~\ref{prop generators t},
the semigroup $\ktT$ is generated by the elements $s_v$ and $\kLp(\kc,d)$
where $v\in\kP$ are vertices, $d\leq\kP$ compact edges, and
$c\in \kQuot$.
Nevertheless, e.g.\ for proving that every relation comes either
from loops or from local relations (obtained after fixing an edge), it is
much easier to replace the generators $\kLp(\kc,d)$ (and the $s_v$)
by new ones being associated to the
new coordinates introduced in Subsection \ref{subsec-newVar}.
\begin{itemize}
\item[($A$)]
For each $a\in A$, we define $\kLp(a):=\krhoA_a$. In particular, these
elements equal certain $s_v$, i.e.\ they are contained in $\ktT$.
\item[($B$)]
For each $b\in B$, we define $\kLp(b):=\krhoB_b$. As in Case $A$, the cluster
contains certain vertices $v\in\kP$, i.e.\ $\kLp(b)=s_v\in\ktT$.
\end{itemize}
At this point, to keep track of the converse,
we have already ensured that all variables $s_v$ are among the
new generators $\kLp(a)$ or $\kLp(b)$ ($a\in A$, $b\in B$).
Moreover, if $d=[v,w)$ is a \se\ half open edge, then there are two cases:
\\[0.5ex]
First, if $w\notin N$ 
(this is equivalent to $(v,w]$ being a \se\ half open edge, too),
then for $\kc\in\kQuot$ with $\langle c,d\rangle\geq 0$ we know
by Definition~\ref{def 53} that
\[
\kLp(c, d)= \lan c,d\ran \,t_{d}+\{\lan
c,w\ran\}s_{w}-\{\lan c,v\ran\}s_{v}
=\big(\lceil\lan\kc,w\ran\rceil - \lceil \lan\kc,v\ran\rceil\big)\cdot \krhoB_b
\in\NN\cdot \krhoB_b
\]
with $d,v,w\in b$ and $b\in B$, hence
$t_d=s_v=s_w=\krhoB_b$.
\\[0.5ex]
Second, if $w\in N$, then $s_w=0$, hence
\[
\kLp(c, d)= \lan c,d\ran \,t_{d} -\{\lan c,v\ran\}s_{v}
=
\big(\lan\kc,w\ran - \lceil \lan\kc,v\ran\rceil\big)\cdot \krhoB_b
\in\NN\cdot \krhoB_b
\]
with $d,v\in b$ and $b\in B$, hence $t_d=s_v=\krhoB_b$.
\\[0.5ex]
Thus, in both cases, the old $\kLp(c, d)$ together with the elements
$s_v$, on the one hand, 
and the new elements $\kLp(a)$ and $\kLp(b)$, on the other, can be 
mutually expressed 
using just semigroup operations.
It remains to treat the non-short edges -- however, here we do not change
anything at all:
\begin{itemize}
\item[($D$)]
For each $d\in D$ and $\kc\in\kQuot$ 
we stay with the usual $\kLp(\kc, d)\in\ktT$. It can be expressed as
$\kLp(\kc, d)=\lan c,d\ran \,\krhoD_{d}
+\{\lan c,w\ran\}\krho_{w}-\{\lan c,v\ran\}\krho_{v}$
where $\krho_{v},\krho_{w}$ are either $0$ (if the corresponding vertex is 
contained in $\kQuotD$), or they are coordinates of some components
from the sets $A$ or $B$.
\end{itemize}
It is now clear that our definitions imply that
\[
\ktT=\langle \kLp(a),\,\kLp(b),\,\kLp(\kc, d)\kst
a\in A,\; b\in B,\; d\in D,\mbox{ and }\kc\in\kQuot\rangle
\]
as a semigroup. That is, the new $\kLp$ still form  a generating system. 
However, we will see in Subsection \ref{subsec-relNewGen} that
their mutual relations are easier to understand.

\subsection{The relations among the new generators}\label{subsec-relNewGen}
The relations among the generators of $\ktT$ defined in 
Subsection \ref{subsec-newGenTt} split into two types.

\subsubsection{The local relations}\label{subsubsec-localRel}
We call relations among the new $\kLp$ \emph{local} if and only if they are relations
with integer coefficients among the elements $\kLp(\kc, d)\in\ktT$ 
for a single, fixed edge $d\in D$ (and finitely many $\kc\in\kQuot$).

\subsubsection{The loop relations}\label{subsubsec-loopRel}
Among the non-local, i.e.\ the global relations, there is a special class of
so-called \emph{loop} relations for any given closed path $\mu\in\Z^{\kpr}$
along the compact edges, 
e.g.\ for the boundary $\mu=\partial F$ of any compact 2-face $F\leq\kP$. For
any $\kc\in\kQuot$, the loop relation $R(\mu,\kc)$
from Subsection~(\ref{subsec-newVar}) among the coordinates
\[ 
\sum_{d\in D} \mu_d\cdot \langle c,d\rangle\cdot \krhoD_{d}
+
\sum_{b\in B} \big\langle c,\sum_{d\in b} \mu_d \,d\big\rangle 
\cdot \krhoB_{b}=0
\]
induces
\[ 
\sum_{d\in D} \mu_d\cdot \Big(\kLp(\kc, d) - \{\lan c,w\ran\}\krho_{w}
+ \{\lan c,v\ran\}\krho_{v}\Big) +
\sum_{b\in B} \big\langle c,\sum_{d\in b} \mu_d \,d\big\rangle 
\cdot \krhoB_{b}=0.
\]
Recalling that $\krho_{v},\krho_{w}$ (and $\krhoB_{b}$) belong to 
the classes $A$ or $B$, i.e.\ {\em not} to class $D$, 
we may replace them by the corresponding
$\kLp(\ldots)$, yielding the loop $\kLp$-relation
\[ 
\sum_{d\in D} \mu_d\cdot \Big(\kLp(\kc, d) - \{\lan c,w\ran\}\kLp(w)
+ \{\lan c,v\ran\}\kLp(v)\Big) +
\sum_{b\in B} \big\langle c,\sum_{d\in b} \mu_d \,d\big\rangle 
\cdot \kLp(b)=0
\]
which we will call $\kRp(\mu,\kc)$.

\begin{proposition}
\label{prop-ttildeRel}
Any integral relation among the elements
$\kLp(a)$, $\kLp(b)$, and $\kLp(\kc, d)$
with $a\in A$, $b\in B$, $d\in D$, and $\kc\in\kQuot$
is an integral linear combination of local relations
{\rm (\ref{subsubsec-localRel})}
and loop relations $\kRp(\mu,\kc)$ from
{\rm (\ref{subsubsec-loopRel})}. 
\end{proposition}

\begin{proof}
Denote by $\kRp$ an arbitrary integral relation like 
\[
\sum_{a\in A}\lambda_a \,\kLp(a) + 
\sum_{b\in B}\lambda_b \,\kLp(b) + \hspace{-0.5em}
\sum_{d\in D,\, \kc\in\kQuot}\hspace{-0.5em}\lambda_{d,c} \,\kLp(\kc, d)
\;=\;0.
\]
Then we use ($A$), ($B$), and ($D$) of Subsection~\ref{subsec-newGenTt},
i.e.\
\begin{eqnarray*}
  \kLp(a)&=&\krhoA_a, \\
  \kLp(b)&=&\krhoB_b,\text{~and}\\
  \kLp(\kc, d)&=&\lan c,d\ran \,\krhoD_{d}+\{\lan c,w\ran\}\krho_{w}-\{\lan c,v\ran\}\krho_{v}
\end{eqnarray*}
to write this as
\[
\sum_{a\in A}\lambda_a \,\krhoA_a +
\sum_{b\in B}\lambda_b \,\krhoB_b + \hspace{-0.5em}
\sum_{d\in D, \,\kc\in\kQuot}\hspace{-0.5em}\lambda_{d,c} \cdot
\Big(\lan c,d\ran \,\krhoD_{d}
+\{\lan c,w\ran\}\krho_{w}-\{\lan c,v\ran\}\krho_{v}\Big)
\;=\;0
\]
with $v=v(d), w=w(d)\in A\cup B$.
Now, we know that this relation among the $\krho$-coordinates is
the sum of certain loop relations $R(\mu,\kc)$ 
from Subsection~(\ref{subsec-newVar}) -- note that
it suffices to take only $\mu:=\partial F$ for some compact 2-faces 
$F\leq \kP$. We denote by $\kRp'$ the corresponding sum of the associated
loop $\kLp$-relations $\kRp(\mu,\kc)$.
By construction, we know that the original $\kRp$ and the sum of 
loop relations $\kRp'$ coincide after being transformed to
relations among the $\krho$-variables.
Modding out the $A$- and $B$-variables, we can then spot a linear combination
of local relations in the sense of
(\ref{subsubsec-localRel}) generating the difference. 
\\[1ex]
Note that the point for 
everything working as it has been said is the triangular structure,
i.e.\ the fact that $\kLp(\kc, d)\mapsto 
\lan c,d\ran \,\krhoD_{d}
+\{\lan c,w\ran\}\krho_{w}-\{\lan c,v\ran\}\krho_{v}$ involves only a single
$d$ and elements of $A\cup B$ where the map $\kLp\mapsto\krho$ is trivial.
\end{proof}

Thus we conclude the proof of  
Theorem \ref{th main}. Indeed, from Corollary \ref{cor loop cor}, Proposition \ref{pro image h g} and Proposition \ref{prop-ttildeRel} it follows the surjectivity of map $o^*$, which proves the theorem.


\begin{thebibliography}{44}




\bibitem{alt}
{K. Altmann: \emph{The versal deformation of an isolated, toric Gorenstein singularity}, Invent. Math. \textbf{128} (1997), 443--479.}

\bibitem{cup}
{K. Altmann: \emph{Infinitesimal deformations and obstructions for toric singularities.} J. Pure Appl. Alg. \textbf{119} (1997), 211--235.}


\bibitem{ka-flip}
{K. Altmann:  \emph{One parameter families containing three-dimensional toric Gorenstein singularities}, Explicit birational geometry of 3-folds, London Math. Soc. Lecture Note Ser., vol. \textbf{281}, Cambridge Univ. Press, Cambridge (2000), 21--50.}

\bibitem{budapest}
{K. Altmann, L. Kastner: \emph{Negative deformations of toric singularities that are smooth in codimension 2}, Deformations of surface singularities, Bolyai Mathematical Society (2013)}.


\bibitem{alt-kol}
{K. Altmann, J. Koll\'ar: \emph{The dualizing sheaf on first order deformations of toric surface singularities}, J. reine angew. Math. \textbf{753} (2019), 137--158.}

\bibitem{klaus}
{K. Altmann, A. B. Sletsj\o e: \emph{Andr\'e-Quillen cohomology of monoid algebras}, J.  Alg. \textbf{210} (1998), 1899--1911.}



\bibitem{a}
{K. Altmann, A. Constantinescu, M. Filip: \emph{Polyhedra, lattice structures, and extensions of semigroups}, arXiv:2004.07377}.

\bibitem{Chr91}
  {J.A. Christophersen: \emph{On the components and discriminant of the versal base space of cyclic quotient singularities}, in Singularity theory and its applications. Springer, Berlin, Heidelberg, 
(1991), 81--92.}

\bibitem{chr-ilt}
{J. Christophersen, N. O. Ilten: \emph{
Hilbert schemes and toric degenerations for low degree Fano threefolds}, J. reine angew. Math. \textbf{717} (2016), 77--100.}


\bibitem{coa}
{T. Coates, A. M. Kasprzyk, T. Prince: \emph{Laurent Inversion}, Pure and Applied Mathematics Quarterly \textbf{15} (2019)}.

\bibitem{cor}
{T. Coates, A. Corti, S. Galkin, V. Golyshev, A. Kasprzyk: \emph{Mirror symmetry and Fano manifolds}, In European congress of mathematics, Eur. Math. Soc. (2013), 285--300}.



\bibitem{mojcor}
{A. Corti, M. Filip, A. Petracci: \emph{Smoothing toric Gorenstein affine $3$-folds, $0$-mutable polynomials and mirror symmetry}, to appear soon}.

\bibitem{jong}
{T. de Jong, D. van Straten: \emph{On the deformation theory of rational surface singularities with reduced fundamental cycle}, J. Alg. Geom. \textbf{3} (1994), 117--172.}

\bibitem{de jong}
{T. de Jong, G. Pfister: \emph{Local analytic geometry. Basic theory and applications,}Advanced lectures in mathematics, Vieweg (2000)}.

\bibitem{Kle79}
  {J.O. Kleppe: \emph{Deformations  of  graded  algebras}, Math. Scand. \textbf{45}, no. 2, (1979), 205--231.}
  
\bibitem{KSB88}
  {J. Koll\'ar, N.I. Shepherd-Barron: \emph{Threefolds and deformations of surface singularities}. Invent. math. \textbf{91}, (1988), 299--338.}



\bibitem{MS05}
{E. Miller, B. Sturmfels: \emph{Combinatorial commutative algebra}, Graduate Texts in Mathematics, Springer Verlag, New York (2005)}.


\bibitem{pri}
{T. Prince: \emph{Smoothing Calabi-Yau toric hypersurfaces using the Gross-Siebert algorithm}, arXiv:1909.02140}.

\bibitem{Ste91}
{J. Stevens: \emph{On the versal deformation of cyclic quotient singularities}, in Singularity theory and its applications. Springer, Berlin, Heidelberg,  (1991), 302--319.}

\end{thebibliography}
\end{document}